\documentclass{article}
\usepackage[utf8]{inputenc}
\usepackage[margin=0.5in]{geometry}
\usepackage{import}
\usepackage{graphicx}
\usepackage{enumerate}
\usepackage{tikz}
\usepackage{amssymb}
\usepackage{amsmath}
\usepackage{amsthm}
\usepackage{fp}	
\usepackage{xcolor,multirow}

\newtheorem{thm}{Theorem}[section]
\newtheorem{lem}[thm]{Lemma}
\newtheorem{prop}[thm]{Proposition}
\newtheorem{cor}[thm]{Corollary}

\numberwithin{equation}{section}

\newcommand{\Aut}{\operatorname{Aut}}

\title{Distinguishing chromatic number of Hamiltonian circulant graphs}
\author{Michael D. Barrus, Jean Guillaume, Benjamin Lantz}

\begin{document}

\maketitle

\begin{abstract}
The distinguishing chromatic number of a graph $G$ is the smallest number of colors needed to properly color the vertices of $G$ so that the trivial automorphism is the only symmetry of $G$ that preserves the coloring. We investigate the distinguishing chromatic number for Hamiltonian circulant graphs with maximum degree at most 4.
\end{abstract}

\section{Introduction and definitions} \label{sec: intro}
The \emph{distinguishing chromatic number} of a graph $G$, introduced by Collins and Trenk~\cite{CollinsTrenk06} and denoted by $\chi_D(G)$, is the minimum number of colors needed for a proper coloring of the vertices of $G$ such that the only automorphism of $G$ that preserves the coloring is the trivial automorphism. Any such coloring using $r$ colors is called an \emph{$r$-distinguishing} coloring, or simply a \emph{distinguishing coloring} if the value of $r$ is unimportant.

The distinguishing chromatic number was introduced as a proper-coloring analogue of the \emph{distinguishing number} of a graph defined by Albertson and Collins~\cite{AlbertsonCollins96}, which measures the difficulty of breaking symmetries in graphs by coloring vertices (i.e., providing a distinguishing coloring), though without requiring a proper coloring. In this paper, all distinguishing colorings will be assumed to be proper colorings.

In~\cite{CollinsTrenk06}, Collins and Trenk observed that both the ordinary chromatic number $\chi(G)$ and the distinguishing number $D(G)$ serve as lower bounds for $\chi_D(G)$. They proved that $\chi_D(G) = |V(G)|$ if and only if $G$ is a complete multipartite graph, and they also determined $\chi_D(G)$ for various classes of graphs $G$. In particular, they showed the following.

\begin{thm}[\cite{CollinsTrenk06}] \label{thm: cycles}
For $n \geq 3$, the distinguishing chromatic number of $C_n$ is given by
\[
\chi_D(C_n) = \begin{cases}
3 & \text{ if }n \in \{3,5\} \text{ or if }n \geq 7;\\
4 & \text{ if }n \in \{4,6\}.
\end{cases}
\]
\end{thm}

The graph $C_6$ is one example where $\chi_D$ can be strictly greater than both the chromatic number and the distinguishing number (both equal 2 for $C_6$).

In this paper we consider circulant graphs, i.e., those undirected graphs with vertices $v_0,\dots,v_{n-1}$ where edges join any two vertices having indices with a difference (in either order) modulo $n$ lying in a given set of positive numbers. If this set of differences is denoted by $D$, then we denote such a graph by $C_n(D)$, and if $D=\{d_1,\dots,d_s\}$, we write the graph as $C_n(d_1,\dots,d_s)$.

For example, the cycle graph $C_n$ is equivalent to the circulant graph $C_n(1)$. 
Note that the notation allows for multiple representations for a single graph, since differences may be computed in the opposite order and are reduced modulo $n$. For example, the cycle $C_n$ is also equivalent to 
$C_n(n-1)$, and we may replace any element $k$ in the set of allowed differences by $n-k$ without changing the graph. Unless otherwise specified, we will assume here that for an $n$-vertex circulant graph, each difference belongs to $\{1,\dots,\lfloor{n/2\rfloor}\}$.

We will extend Theorem~\ref{thm: cycles} by determining the distinguishing chromatic number for various classes of Hamiltonian circulant graphs with maximum degree at most 4. As we will see, in most cases these graphs are similar enough to cycles that the distinguishing chromatic number is 3, and in no infinite family does the number ever rise higher than 5. In particular, we show that for any tetravalent graph $C_n(1, k),$ where $k \neq n/2, n/2 -1,$ and $(n, k) \neq (10, 3),$ the distinguishing chromatic number is at most one more than the chromatic number. As a prelude, here is a summary of our main results. 

\begin{thm} \label{thm:summary}
    Let $k,n$ be positive integers such that $2 \leq k \leq \lfloor n/2 \rfloor$. The distinguishing chromatic number of $C_n(1,k)$ is given by \[\chi_D(C_n(1,k)) = \begin{cases}
    n & \text{ if $(n, k) = (4, 2), (5, 2), (6, 2), (6, 3), (8, 3);$}\\
    5 & \text{ if $(n, k) = (10, 3);$}\\
    4 & \text{ if $(n, k) = (15, 4), (13, 5);$}\\
    3 & \text{ if $k =n/2$ and $n \geq 8;$}\\
    5 & \text{ if $k =n/2 -1$ and $n \geq 10;$ }\\
    4 & \text{ if $k=2$ or $k= (n-1)/2,$ and $n \geq 7;$}\\
    3 & otherwise.
    \end{cases}\]
\end{thm}

By otherwise, we mean all tetravalent circulant graphs $C_n(1, k)$ such that $k \neq 2, (n-1)/2,$ and $n/2 -1,$ and $(n, k) \neq (10, 3), (15, 4), (13, 5), (5, 2),$ and $(8, 3).$ The circulant graphs $C_n(1,k)$ such that $(n, k)= (4, 2), (5, 2), (6, 3), (8, 3)$ are complete graphs or complete bipartite graphs. Thus, results hold by~\cite{CollinsTrenk06}. As for graph $C_6(1, 2),$ the following arguments can be used to show that the distinguishing number is 6: vertices $v_iv_{i+2}v_{i-2}$ form an induced triangle and antipodal pairs $v_iv_{i+3}, v_{i+2}v_{i+5}, v_{i-2}v_{i-5}$ must have different labels. The remaining results in Theorem~\ref{thm:summary} are shown in Theorems~\ref{thm: Mobius}, ~\ref{thm: even n and k= n/2-1},~\ref{thm: k= 2},~\ref{thm: n even, k even},~\ref{thm:both odd},~\ref{thm:n odd and k even}, and Propositions~\ref{thm: n=10, k=3},~\ref{thm: nondihedral}.

The following sections are organized as follows. In Section~\ref{sec: isos}, we recall facts about isomorphisms and proper colorings of circulant graphs. In Section~\ref{sec: trivalent}, we determine $\chi_D$ for the trivalent Hamiltonian circulant graphs, also known as the M\"obius ladders. In Section~\ref{sec: automorphism groups}, we discuss the automorphism of the tetravalent graphs $C_n(1,k).$ In Section~\ref{sec: k = n/2-1}, we give an optimal distinguishing proper coloring of $C_n(1, n/2 -1).$ In Sections~\ref{sec: chi plus 1},~\ref{sec: dihedral symmetries}, and~\ref{sec: k^2 = pm 1}, we prove that the distinguishing chromatic number of tetravalent graphs $C_n(1, k),$  where $k \neq n/2, n/2 -1,$ and $(n, k) \neq (10, 3),$ is at most 1 more than the ordinary chromatic number.

Throughout this paper, we will use $V(G)$ and $E(G)$ to denote the vertex and edge sets of a graph $G$.

\section{Isomorphisms and colorings of circulant graphs} \label{sec: isos}
The results of this paper will deal principally with circulant graphs having the form $C_n(1,k)$ for various $n,k$; the inclusion of 1 as one of the differences forces the graph to be Hamiltonian. Though not investigated in this paper, observe that circular graphs with different differences sets can also be Hamiltonian, as shown by the example $C_6(2,3)$, which is isomorphic to the triangular prism and is not isomorphic to  $C_6(1,k)$ for any $k$. 

\begin{figure} 
\centering
    \begin{tikzpicture}
        \node at (90:3) {$C_{7}(1,2)$};
        \node (0) [draw,shape=circle, fill=black, scale=.4] at (90:2) {};
        \node at (90:2.3) {$v_0$};
        \node (1) [draw,shape=circle, fill=black, scale=.4] at (39:2) {};
        \node at (39:2.3) {$v_1$};
        \node (2) [draw,shape=circle, fill=black, scale=.4] at (347:2) {};
        \node at (347:2.3) {$v_2$};
        \node (3) [draw,shape=circle, fill=black, scale=.4] at (296:2) {};
        \node at (296:2.3) {$v_3$};
        \node (4) [draw,shape=circle, fill=black, scale=.4] at (244:2) {};
        \node at (244:2.3) {$v_4$};
        \node (5) [draw,shape=circle, fill=black, scale=.4] at (192:2) {};
        \node at (192:2.3) {$v_5$};
        \node (6) [draw,shape=circle, fill=black, scale=.4] at (141:2) {};
        \node at (141:2.3) {$v_6$};
        \draw (0)--(1)--(2)--(3)--(4)--(5)--(6)--(0)--(2)--(4)--(6)--(1)--(3)--(5)--(0);
    \end{tikzpicture}
    \begin{tikzpicture}
        \node at (90:3) {$C_{7}(1,3)$};
        \node (0) [draw,shape=circle, fill=black, scale=.4] at (90:2) {};
        \node at (90:2.3) {$v_0$};
        \node (3) [draw,shape=circle, fill=black, scale=.4] at (39:2) {};
        \node at (39:2.3) {$v_3$};
        \node (6) [draw,shape=circle, fill=black, scale=.4] at (347:2) {};
        \node at (347:2.3) {$v_6$};
        \node (2) [draw,shape=circle, fill=black, scale=.4] at (296:2) {};
        \node at (296:2.3) {$v_2$};
        \node (5) [draw,shape=circle, fill=black, scale=.4] at (244:2) {};
        \node at (244:2.3) {$v_5$};
        \node (1) [draw,shape=circle, fill=black, scale=.4] at (192:2) {};
        \node at (192:2.3) {$v_1$};
        \node (4) [draw,shape=circle, fill=black, scale=.4] at (141:2) {};
        \node at (141:2.3) {$v_4$};
        \draw (0)--(1)--(2)--(3)--(4)--(5)--(6)--(0)--(3)--(6)--(2)--(5)--(1)--(4)--(0);
    \end{tikzpicture}
    \caption{Isomorphic graphs $C_7(1,2)$ and $C_7(1,3)$.}
    \label{fig: isos}
\end{figure}
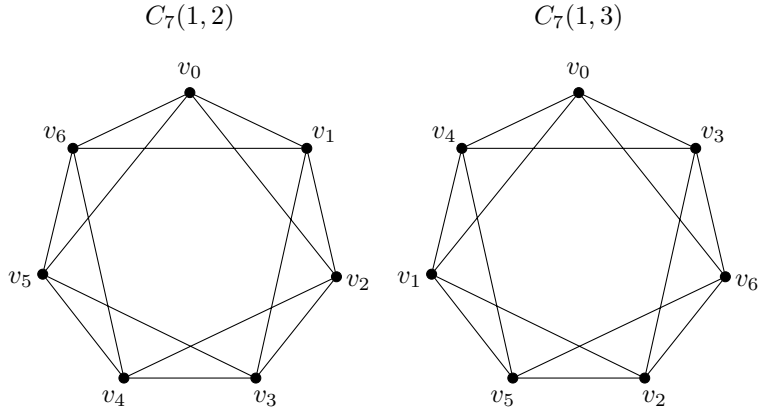

Recall from Section 1 that isomorphic circulant graphs can have multiple representations using distinct sets of differences. Besides the examples given there, observe that the graph $C_5$ is isomorphic to both  $C_5(1)$ and to $C_5(2)$, and Figure~\ref{fig: isos} shows that $C_{7}(1,2)$ is isomorphic to $C_{7}(1,3)$. We may simplify cases in what follows, and extend our results about graphs $C_n(1,k)$ to isomorphic trivalent or tetravalent circulant graphs not expressed in this way by recalling a result of \'Ad\'am~\cite{Adam67}.

\begin{thm}[\cite{Adam67}] \label{thm: iso}
If $\gcd(n,p)=1$, then $C_n(a_1, \cdots, a_t) \cong C_n(pa_1, \cdots, pa_t)$, where multiplication is performed modulo $n$.
\end{thm}

We specialize this result to the graphs of the form $C_n(1,k)$.

\begin{cor}
If either $a$ or $b$ is relatively prime to $n$, then $C_n(a,b) = C_n(1,k)$ for some $k \in \{1,\dots,\lfloor{n/2}\rfloor\}$.
\end{cor}
\begin{proof}
Suppose without loss of generality that $\gcd(a,n)=1$. An elementary result of number theory shows that there exists $p$ in $\{1,\dots,n-1\}$ such that $ap \equiv 1 \pmod{n}$ and $p$ is relatively prime to $n$. Hence $C_n(a,b) \cong C_n(pa,pb) \cong C_n(1,k)$, where $k$ is either $pb$ or $n-pb$ modulo $n$, whichever belongs to $\{1,\dots,\lfloor{n/2}\rfloor\}$.
\end{proof}

\begin{cor} \label{cor: iso} If $n=k\ell\pm 1$, then $C_n(1,k) \cong C_n(1,\ell)$.
\end{cor}
\begin{proof}
If $n=k\ell\pm 1$, then $\ell$ is relatively prime to $n$, as is $n-\ell$. Letting $p=n-\ell$ in Theorem~\ref{thm: iso} shows that $C_n(1,k) \cong C_n(n-\ell, \pm 1) \cong C_n(1,\ell)$.
\end{proof}

We recall now a few results about the chromatic number $\chi(G)$ of circulant graphs $G$. The following result was conjectured by Collins, Fisher, and Hutchinson (see~\cite{CollinsEtAl98, Fisher98} as cited in~\cite{YehZhu03}) and proved by Yeh and Zhu~\cite{YehZhu03}; see also~\cite{GobelNeutel00}, \cite{Heuberger03}, and~\cite{NicolosoPietropaoli07}.

\begin{thm} \label{thm: chromatic num}
    Let $k,n$ be positive integers such that $2 \leq k \leq \lfloor n/2 \rfloor$. The chromatic number of $C_n(1,k)$ is given by \[\chi(C_n(1,k)) = \begin{cases}
    2 & \text{ if $k$ is odd and $n$ is even;}\\
    4 & \text{ if $k = 2$ or $k = (n-1)/2$, and $n \neq 5$ and $3 \nmid n$;}\\
    4 & \text{ if $k=5$ and $n=13$;}\\
    5 & \text{ if $k=2$ and $n=5$;}\\
    3 & otherwise.
    \end{cases}\]
\end{thm}

\section{Trivalent circulant graphs} \label{sec: trivalent}
When $n$ is even, the circulant graph $C_n(1,n/2)$ is a trivalent graph also known as the \emph{M\"obius Ladder} due to a drawing as a M\"obius band of 4-cycles; see Figure~\ref{fig: Mobius} which draws $C_8(1,4)$ in two ways. These are the unique trivalent circulant graphs $C_n(1,k)$.

\begin{figure} 
    \centering
    \begin{tikzpicture}
    \node (0) [draw,shape=circle, fill=black, scale=.4] at (90:2) {};
    \node at (90:2.3) {$v_0$};
    \node (1) [draw,shape=circle, fill=black, scale=.4] at (45:2) {};
    \node at (45:2.3) {$v_1$};
    \node (2) [draw,shape=circle, fill=black, scale=.4] at (0:2) {};
    \node at (0:2.3) {$v_2$};
    \node (3) [draw,shape=circle, fill=black, scale=.4] at (315:2) {};
    \node at (315:2.3) {$v_3$};
    \node (4) [draw,shape=circle, fill=black, scale=.4] at (270:2) {};
    \node at (270:2.3) {$v_4$};
    \node (5) [draw,shape=circle, fill=black, scale=.4] at (225:2) {};
    \node at (225:2.3) {$v_5$};
    \node (6) [draw,shape=circle, fill=black, scale=.4] at (180:2) {};
    \node at (180:2.3) {$v_6$};
    \node (7) [draw,shape=circle, fill=black, scale=.4] at (135:2) {};
    \node at (135:2.3) {$v_7$};
    \draw (0)--(1)--(2)--(3)--(4)--(5)--(6)--(7)--(0)--(4);
    \draw (1)--(5);
    \draw (2)--(6);
    \draw (3)--(7);
    \end{tikzpicture}
    \begin{tikzpicture}
    \node (0) [draw,shape=circle, fill=black, scale=.4] at (90:2) {};
    \node at (90:2.3) {$v_0$};
    \node (4) [draw,shape=circle, fill=black, scale=.4] at (90:1) {};
    \node at (90:0.7) {$v_4$};
    \node (5) [draw,shape=circle, fill=black, scale=.4] at (0:2) {};
    \node at (0:2.3) {$v_5$};
    \node (1) [draw,shape=circle, fill=black, scale=.4] at (0:1) {};
    \node at (0:0.7) {$v_1$};
    \node (6) [draw,shape=circle, fill=black, scale=.4] at (270:2) {};
    \node at (270:2.3) {$v_6$};
    \node (2) [draw,shape=circle, fill=black, scale=.4] at (270:1) {};
    \node at (270:0.7) {$v_2$};
    \node (7) [draw,shape=circle, fill=black, scale=.4] at (180:2) {};
    \node at (180:2.3) {$v_7$};
    \node (3) [draw,shape=circle, fill=black, scale=.4] at (180:1) {};
    \node at (180:0.7) {$v_3$};
    \draw (0)--(1)--(2)--(3)--(4)--(5)--(6)--(7)--(0)--(4);
    \draw (1)--(5);
    \draw (2)--(6);
    \draw (3)--(7);
    \end{tikzpicture}
    \caption{The graph $C_8(1,4)$ drawn as a M\"obius ladder.}
    \label{fig: Mobius}
\end{figure}
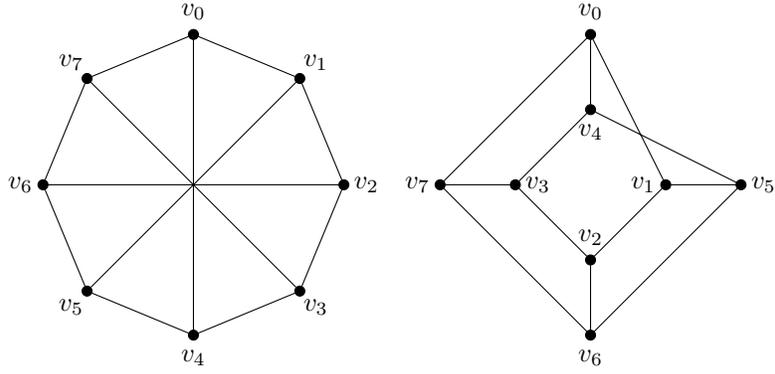

\begin{thm} \label{thm: Mobius}
For even integers $n \geq 4$, \[\chi_D(C_n(1,n/2)) = \begin{cases}
4 & \text{ if }n=4;\\
6 & \text{ if }n=6;\\
3 & \text{ if }n \geq 8.
\end{cases}\]
\end{thm}
\begin{proof}
When $n$ is 4 or 6, the graph $C_n(1,n/2)$ is isomorphic to $K_4$ or $K_{3,3}$, respectively. By the result of Collins and Trenk~\cite{CollinsTrenk06}, distinguishing colorings of these complete multipartite graphs require as many colors as there are vertices in the graph.

For even $n \geq 8$, it has been shown (see~\cite{MortezaMirafzal}, for instance) that the M\"obius ladder with $n$ vertices has the dihedral group of order $2n$ as its automorphism group. We exhibit a distinguishing proper coloring using 3 colors based on whether or not $n$ is a multiple of 4.

If $n=4t$, where $t>1$, then the M\"obius ladder is not bipartite and hence a proper coloring will require at least 3 colors. Assign colors to the vertices as shown here (assigning the colors shown to the vertices in order of their subscripts):
\begin{center}
\begin{tabular}{c|c|c|c|c|c|c|c|c|c|}

Colors of $v_0,\dots,v_{n/2-1}$   & 1 & 2 & 1 & 2 & 1 & 2 & $\cdots$ & 1 & 2 \\
\hline
Colors of $v_{n/2},\dots,v_{n-1}$ & 3 & 1 & 2 & 1 & 2 & 1 & $\cdots$ & 2 & 3 
\end{tabular}
\end{center}
Observe that each vertex is adjacent to vertices whose colors precede and follow its color in the table, as well as to the vertex whose color appears in the same position on the opposite row. Thus it is easy to verify that this is a proper coloring. (In fact, this coloring arises via a greedy coloring of the vertices $v_0,\dots,v_{n-1}$ in order.) Note that a coloring-preserving automorphism must permute the vertices with color 3, and one of these two vertices has two neighbors with color 1 while the other has only one neighbor with color 1. Hence the vertices with color 3 must be fixed under any color-preserving automorphism. The only dihedral symmetry fixing two non-antipodal vertices is the identity symmetry, so this coloring is a distinguishing coloring.

If $n=4t+2$, where $t>1$, then $C_n(1,n/2)$ is a bipartite graph. The unique partition of its vertices into two color classes allows for many dihedral symmetries, so a distinguishing proper coloring must use at least 3 colors. We obtain one by changing to color 3 the colors of a pair of nonadjacent vertices formerly colored 1 and 2 in a proper 2-coloring. One example is shown here:
\begin{center}
\begin{tabular}{c|c|c|c|c|c|c|}

Colors of $v_0,\dots,v_{n/2-1}$   & 3 & 2 & 1 & 2 & $\cdots$ & 1 \\
\hline
Colors of $v_{n/2},\dots,v_{n-1}$ & 2 & 1 & 3 & 1 & $\cdots$ & 2 
\end{tabular}
\end{center}
As before, we observe that a color-preserving automorphism must permute the vertices of color 3. Since one is adjacent only to vertices of color 1, while the other is adjacent only to vertices of color 2, the automorphism fixes these vertices. As before, the identity is thus the only color-preserving automorphism, and this is a distinguishing coloring.
\end{proof}

\section{Automorphism groups} \label{sec: automorphism groups}
For the remainder of the paper we will deal with the tetravalent graphs $C_n(1,k)$ where $1<k<n/2$. Because distinguishing colorings on graphs ``break'' nontrivial symmetries, this section will review some facts about automorphism groups of circulant graphs.

The following theorem of Poto\u{c}nik and Wilson~\cite{PotocnikWilson20} will be key to our organization. A graph is \emph{edge-transitive} if any edge may be carried to any other edge by some automorphism, and \emph{dart-transitive} if any edge may be mapped to any other edge, with the endpoint images specified, by some automorphism.

\begin{thm}[\cite{PotocnikWilson20}] \label{thm: PW automorphisms}
If $G$ is a tetravalent edge-transitive circulant graph with $n$ vertices, then it is dart-transitive and either:
\begin{enumerate}
    \item[\textup{(1)}] $G$ is isomorphic to $C_n(1,k)$ for some $a$ such that $k^2  \equiv \pm 1 \pmod{n}$, or 
    \item[\textup{(2)}] $n$ is even, 
    and $G$ is isomorphic to $C_{2m}(1,m+1)$ where $m=n/2$.
\end{enumerate}
\end{thm}

As we will see later as we color the graphs $C_n(1,k)$, the converse to Theorem~\ref{thm: PW automorphisms} is true. We will provide an optimal distinguishing proper coloring for the graphs in (2) in Section~\ref{sec: k = n/2-1}, noting that $C_{2m}(1,m+1)$ is isomorphic to $C_n(1,n/2-1)$, and say no more about these graphs here. We will discuss coloring the graphs in (1) in Section~\ref{sec: k^2 = pm 1} after a few comments about these graphs at the end of this section.

What about graphs $C_n(1,k)$ that are not edge-transitive? As we will see, in that case $\Aut(C_n(1,k)$ is isomorphic to a dihedral group of order $2n$.

Let $E_1$ denote the set of edges of $C_n(1,k)$ joining vertices having indices differing by 1 modulo $n$, and similarly let $E_k$ denote the set of edges of $C_n(1,k)$ joining vertices  whose indices' difference is $k$ modulo $n$. Observe that an automorphism $\phi:V(C_n(1,k)) \to V(C_n(1,k))$ induces a permutation on the edge set of $C_n(1,k)$ that ``carries'' edge $v_iv_j$ to edge $\phi(v_i)\phi(v_j)$ for any $i,j$.

\begin{thm} \label{thm: E_1 and E_k permutations}
If $k \neq n/2-1$ and $(n,k) \neq (10,3)$, then any automorphism of $C_n(1,k)$ either carries every edge in $E_1$ to an edge in $E_1$ or carries every edge in $E_1$ to an edge in $E_k$.
\end{thm}
\begin{proof}
The result can be verified directly for $n < 7$, so assume that $n \geq 7$.

Consider an edge of $E_1$ in $C_n(1,k)$; by symmetry we may assume that it is $v_0v_1$. Note that this edge belongs to the two 4-cycles $v_0 v_{1} v_{k+1} v_{k}$ and $v_0 v_{1} v_{1-k} v_{-k}$; these are distinct 4-cycles unless $k=n/2$, in which case they coincide. Note that for each edge in $E_k$ incident with either $v_0$ or $v_{1}$, there is a 4-cycle containing that edge and $v_0 v_{1}$. We claim now that no 4-cycle contains both $v_0 v_{1}$ and an incident edge from $E_1$. Indeed, for a 4-cycle to contain the path $v_0 v_{1} v_{2}$, both $v_0$ and $v_{2}$ must have a common neighbor $v_j$, where $j \in \{-k,-1,k\} \cap \{2-k,3,k+2\}$. Setting each element of the first set equal to each element of the second set and recalling that $n \geq 7$ and that $1 < k < n/2$, we see that the existence of a common neighbor $v_j$ implies that either $k = n/2 - 1$, a contradiction to our hypothesis, or $k=3$. A similar argument and conclusion holds if $C_n(1,k)$ has a 4-cycle containing the path $v_{-1} v_0 v_{1}$.

Let us assume for now that $k \neq 3$. Since the property of inclusion of a pair of edges in some induced 4-cycle is preserved by an automorphism, it follows that if an edge from $E_1$ is carried by an automorphism to an edge from $E_1$, the same must be true for the images of its incident edges from $E_1$. Working inductively outward from the first edge, we see that each edge of $E_1$ is then carried to an edge in $E_1$, and edges from $E_k$ are forced to be carried to edges in $E_k$. Conversely, if an edge from $E_1$ is carried by an automorphism of $C_n(1,k)$ to an edge in $E_k$, then no edge from $E_1$ can be carried to an edge from $E_1$, which forces all edges from $E_1$ to be carried to edges from $E_k$ and vice versa.

If instead $k=3$, observe directly that in each of $C_n(1,3)$ for $7 \leq n \leq 12$ except for $n \in \{8,10\}$, every automorphism of $C_n(1,3)$ either carries all edges in $E_1$ to edges in $E_1$ or carries all edges in $E_1$ to $E_3$, as claimed. Assume now that $k=3$ and $n > 12$. Note that the size of $n$ implies that each edge in $E_1$ belongs to exactly five 4-cycles; for $v_0v_1$ these cycles are $v_0v_1v_{n-2}v_{n-3}$, $v_0v_1v_{n-2}v_{n-1}$, $v_0v_1v_{2}v_{n-1}$, $v_0v_1v_{2}v_{3}$, and $v_0v_1v_{4}v_{3}$. In contrast each edge in $E_3$ belongs to exactly three 4-cycles; for $v_0v_3$ these are $v_0v_3v_2v_{n-1}, v_0v_3v_2v_1, v_0v_3v_4v_1$. Since the inclusion of an edge in a 4-cycle is preserved under the image of a graph automorphism $\phi$, any such map $\phi$ induces a permutation of the edges in $E_1$ and a permutation of the edges of $E_k$.
\end{proof}

We arrive at our result; let $D_n$ be the dihedral group of order $2n$, the symmetry group of a regular $n$-gon.

\begin{cor} \label{cor: dihedral group}
If the graph $C_n(1,k)$, where $1 < k < n/2$, satisfies neither $k^2 \equiv \pm 1 \pmod{n}$ nor $k=n/2-1$, then $\Aut(C_n(1,k)) \cong D_{n}$.
\end{cor}
\begin{proof}
The dihedral group $D_n$ is always isomorphic to the subgroup of $\operatorname{Aut}(C_n(1,k)$) consisting of automorphisms that carry $E_1$ to $E_1$ and $E_k$ to $E_k$. These symmetries act transitively on $E_1$ and $E_k$, so if $\Aut(C_n(1,k))$ were to contain any automorphism carrying an edge from $E_1$ to $E_k$, or vice versa, then the compositions of this automorphism with suitable dihedral symmetries would yield automorphisms causing $\Aut(C_n(1,k))$ to be edge-transitive and hence implying that $C_n(1,k)$ satisfies conclusion (1) or (2) in Theorem~\ref{thm: PW automorphisms}, a contradiction to our hypothesis.
\end{proof}

We conclude this section by describing the automorphism groups of the edge-transitive graphs in (1) in Theorem~\ref{thm: PW automorphisms}, those graphs $C_n(1,k)$ for which $k^2 \equiv \pm 1 \pmod{n}$. These groups are more elaborate than dihedral groups, but not by much.

\begin{thm} \label{thm: Aut when k^2 equiv pm 1}
If $k$ and $n$ are integers satisfying $1 < k < n/2$ and $k^2 \equiv \pm 1 \pmod{n}$, then $|\Aut(C_n(1,k))|=4n$, and for any two edges $v_a v_b, v_s v_t$ in the graph, there is a unique automorphism sending $v_a$ to $v_s$ and $v_b$ to $v_t$.
\end{thm}
\begin{proof}
By symmetry it suffices to show that for any edge $v_s v_t$ in the graph, there is a unique automorphism sending $v_0$ to $v_s$ and $v_1$ to $v_t$. We claim that this map  is $\phi_{st}$ on $\{v_0,\dots,v_{n-1}\}$ given by $\phi_{st}(v_i) = v_{s+(t-s)i}$ (with all operations performed modulo $n$). To see that is indeed an automorphism, note that the pair $v_x,v_y$ of vertices in $C_n(1,k)$ is adjacent if and only if $x-y$ is congruent modulo $n$ to either $\pm 1$ or $\pm k$. Now the difference in the indices of $\phi_{st}(v_x)$ and $\phi_{st}(v_y)$ is \begin{equation} \label{eq: automorphism proof} s+(t-s)x - s - (t-s)y = (t-s)(x-y).\end{equation} Since $v_s v_t$ is an edge in $C_n(1,k)$, we know that $t-s \in \{\pm 1, \pm k\}$. Recalling that $k^2 \equiv \pm 1 \pmod{n}$, it is straightforward to check that $x-y \in \{\pm 1,\pm k\}$ if and only if $(t-s)(x-y) \in \{\pm 1, \pm k\}$, so $\phi_{st}$ is an element of $\Aut(C_n(1,k))$.

To finish our proof, we show the uniqueness of the automorphism respectively mapping $v_0$ and $v_1$ to $v_s$ and $v_t$. Let $\rho$ be any automorphism of $C_n(1,k)$ mapping vertex $v_0$ to $v_s$ and $v_1$ to $v_t$. Taking $i \in \{1,k\}$ to be the index such that $v_sv_t \in E_i$, Theorem~\ref{thm: E_1 and E_k permutations} implies that $\rho$ carries all edges in $E_1$ to edges in $E_i$, so since $v_2$ is adjacent to $v_1$ along an edge from $E_1$, its image $\rho(v_2)$ is adjacent to $v_t$ along an edge from $E_i$. Since $\rho(v_2) \neq \rho(v_0)$, and $v_t$ only has two neighbors along edges from $E_i$, $\rho(v_2)$ is uniquely determined; we have $\rho(v_2) = \phi_{st}(v_2)$. Continuing inductively through all the vertices $v_2,\dots,v_{n-1}$, we conclude that $\rho = \phi_{st}$, as desired.
\end{proof}

\section{The graphs $C_n(1,n/2-1)$} \label{sec: k = n/2-1}
In this section we give an optimal distinguishing proper coloring of the edge-transitive graphs described in (2) in Theorem~\ref{thm: PW automorphisms}. 
In these graphs $n$ is even, and each vertex $v_i$ is adjacent to the same neighbors to which its \emph{antipodal vertex} $v_{i + n/2}$ is. It follows that $C_n(1,n/2-1)$ is isomorphic to the \emph{wreath graph} $W(n/2,2)$; in general, the wreath graph $W(a,b)$ has $ab$ vertices, partitioned into $a$ independent sets $I_1,\dots,I_a$, each of size $b$, with the vertices in each set $I_i$ being adjacent to all vertices in $I_{i-1}$ and $I_{i+1}$ (with operations in subscripts performed modulo $a$). The graph $C_{12}(1,5)$ and its interpretation as $W(6,2)$ are illustrated in Figure~\ref{fig: C12(1,5)}.

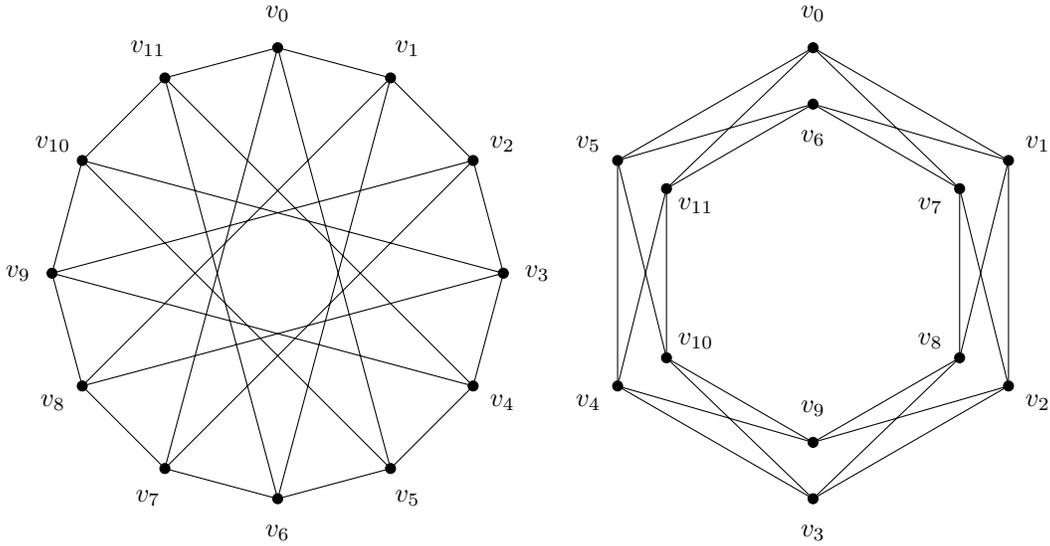
\begin{figure} 
\centering
\begin{tikzpicture}[scale=1.5]
     \node  at (90:2.3) {$v_0$};
     \node  at (60:2.3) {$v_1$};
     \node   at (30:2.3) {$v_2$};
     \node   at (0:2.3) {$v_3$};
     \node   at (330:2.3) {$v_4$};
     \node   at (300:2.3) {$v_5$};
     \node   at (270:2.3) {$v_6$};
     \node   at (240:2.3) {$v_7$};
     \node   at (210:2.3) {$v_8$};
     \node   at (180:2.3) {$v_9$};
     \node   at (150:2.3) {$v_{10}$};
     \node  at (120:2.3) {$v_{11}$};
     \node (0) [draw,shape=circle, fill=black, scale=.4] at (90:2) {};
     \node (1) [draw,shape=circle, fill=black, scale=.4] at (60:2) {};
     \node (2) [draw,shape=circle, fill=black, scale=.4] at (30:2) {};
     \node (3) [draw,shape=circle, fill=black, scale=.4] at (0:2) {};
     \node (11) [draw,shape=circle, fill=black, scale=.4] at (120:2) {};
     \node  (10) [draw,shape=circle, fill=black, scale=.4] at (150:2) {};
     \node (9) [draw,shape=circle, fill=black, scale=.4] at (180:2) {};
     \node (8) [draw,shape=circle, fill=black, scale=.4] at (210:2) {};
     \node (7) [draw,shape=circle, fill=black, scale=.4] at (240:2) {};
     \node (6) [draw,shape=circle, fill=black, scale=.4] at (270:2) {};
     \node (5) [draw,shape=circle, fill=black, scale=.4] at (300:2) {};
     \node (4)  [draw,shape=circle, fill=black, scale=.4] at (330:2) {};
     \draw (0)--(1)--(2)--(3)--(4)--(5)--(6)--(7)--(8)--(9)--(10)--(11)--(0);
     \draw (0)--(5)--(10)--(3)--(8)--(1)--(6)--(11)--(4)--(9)--(2)--(7)--(0);
\end{tikzpicture}
    \begin{tikzpicture}[scale=1.5]
     \node  at (90:2.3) {$v_0$};
     \node  at (30:2.3) {$v_1$};
     \node   at (330:2.3) {$v_2$};
     \node   at (270:2.3) {$v_3$};
     \node   at (210:2.3) {$v_4$};
     \node   at (150:2.3) {$v_5$};
     \node   at (90:1.2) {$v_6$};
     \node   at (30:1.2) {$v_7$};
     \node   at (330:1.2) {$v_8$};
     \node   at (270:1.2) {$v_9$};
     \node   at (210:1.2) {$v_{10}$};
     \node  at (150:1.2) {$v_{11}$};
     \node (0) [draw,shape=circle, fill=black, scale=.4] at (90:2) {};
     \node (1) [draw,shape=circle, fill=black, scale=.4] at (30:2) {};
     \node (2) [draw,shape=circle, fill=black, scale=.4] at (330:2) {};
     \node (3) [draw,shape=circle, fill=black, scale=.4] at (270:2) {};
     \node (11) [draw,shape=circle, fill=black, scale=.4] at (150:1.5) {};
     \node  (10) [draw,shape=circle, fill=black, scale=.4] at (210:1.5) {};
     \node (9) [draw,shape=circle, fill=black, scale=.4] at (270:1.5) {};
     \node (8) [draw,shape=circle, fill=black, scale=.4] at (330:1.5) {};
     \node (7) [draw,shape=circle, fill=black, scale=.4] at (30:1.5) {};
     \node (6) [draw,shape=circle, fill=black, scale=.4] at (90:1.5) {};
     \node (5) [draw,shape=circle, fill=black, scale=.4] at (150:2) {};
     \node (4)  [draw,shape=circle, fill=black, scale=.4] at (210:2) {};
     \draw (0)--(1)--(2)--(3)--(4)--(5)--(6)--(7)--(8)--(9)--(10)--(11)--(0);
     \draw (0)--(5)--(10)--(3)--(8)--(1)--(6)--(11)--(4)--(9)--(2)--(7)--(0);  
    \end{tikzpicture}
    \caption{Two drawings showing $C_{12}(1,5) \cong W(6,2)$.}
    \label{fig: C12(1,5)}
\end{figure}

The automorphism group of $C_n(1,k)$, when $k=n/2-1$, contains ``dihedral'' symmetries interpreted as acting on the graph when drawn as on the left in Figure~\ref{fig: C12(1,5)}. In addition, however, there are $n/2$ symmetries that interchange two vertices having the same neighborhood and fix all other vertices. (Vertices with the same neighborhood are called \emph{twins}.) Hence $\Aut(C_n(1,k))$ has many more than $2n$ automorphisms. Intuitively, this imposes more restrictions on distinguishing colorings (in particular, vertices having the same neighborhoods must receive distinct colors). Hence we may expect to need more colors than we do with $C_n$ to break all symmetries of $C_n(1,k)$ with a proper coloring, and indeed this is the case.

\begin{thm}\label{thm: even n and k= n/2-1}
For even integers $n \geq 8$, \[\chi_D(C_n(1,n/2-1)) = \begin{cases}
6 & \text{ if }n = 6;\\
8 & \text{ if }n = 8;\\
5 & \text{ if }n \geq 10.
\end{cases}\]
\end{thm}
\begin{proof}
It is easy to see that $C_6(1,2)$ is isomorphic to $K_{2,2,2}$ and that $C_8(1,3)$ is isomorphic to $K_{4,4}$, which are both complete multipartite graphs. This implies (see Section~\ref{sec: intro}) that $\chi_D(C_6(1,2)) = 6$ and $\chi_D(C_8(1,3)) = 8$. Suppose henceforth that $n \geq 10$. As mentioned above, a distinguishing coloring must assign different colors to twins, and for $n \geq 10$ the vertex set of $C_n(1,n/2-1)$ is partitioned into $n/2$ such twin pairs. Since our coloring is to be proper, it cannot use the same color on two vertices from ``consecutive'' pairs $\{v_i,v_{i+n/2}\}$ and $\{v_{i+1},v_{i+1+n/2}\}$. Such a coloring must then use at least 4 colors, but 4 colors are not enough for a proper coloring if $n/2$ is odd, and if $n/2$ is even, the 4-coloring of $C_n(1,n/2-1)$ (which is unique up to permutations of the colors) admits the color-preserving automorphism defined by $v_i \mapsto v_{i+2}$ for all $i$. 

Hence at least 5 colors are necessary for a distinguishing coloring. Note now that the pairs of colors assigned to pairs of twins naturally correspond to vertices in the Kneser graph $KG_{c,2}$, where $c$ is the number of colors used in the coloring. (Recall that the Kneser graph $KG_{p,q}$ is the graph whose vertices are the $q$-element subsets of a set of $p$ elements, with edges joining vertices corresponding to disjoint subsets.) Moving from one pair of 
colored twins to the consecutively following pair and noting the colors used corresponds to moving along edges in $KG_{c,2}$, and the overall coloring of $C_n(1,n/2-1)$ corresponds to a closed walk in $KG_{c,2}$. For a proper coloring, any such closed walk will do, and a useful walk in $KG_{5,2}$ (i.e., the Petersen graph) consists of the pairs \[
\begin{cases}
12,34,15,\underline{23,45} & \text{if $n/2$ is odd,}\\
12,34,25,14,\underline{23,45} & \text{if $n/2$ is even;}
\end{cases}\]
here the underlined pairs are repeated as necessary to produce $n/2$ pairs of colors. Note that in each corresponding coloring of $C_n(1,n/2-1)$, there is at most one pair of twins receiving a color pairs from $\{12,34,15,25,14\}$. Any automorphism maps a pair of twins to a pair of twins, and each vertex in $C_n(1,n/2-1)$ belongs to a unique pair of twins when $n/2 \geq 5$. Color-preserving automorphisms likewise preserve the colors on pairs of twins, so the vertices in the twin pairs whose colors were just listed must be fixed any such automorphism. By inductively moving to neighboring twin pairs along the wreath, we see that every other vertex must be fixed, so the only color-preserving automorphism is the identity, as desired.
\end{proof}

\section{A general upper bound} \label{sec: chi plus 1}
Having determined the symmetries of $C_n(1,k)$ when $k \neq n/2 -1$ in Section~\ref{sec: automorphism groups}, for the remaining sections we turn to distinguishing colorings. In this section we show that in many cases, the distinguishing chromatic number of $C_n(1,k)$ is at most 1 more than the ordinary chromatic number. This will allow us to exactly determine the distinguishing chromatic number whenever $C_n(1,k)$ is bipartite.

Our first result is useful in ``breaking'' symmetries in edge-transitive graphs $C_n(1,k)$.
\begin{lem} \label{lem: one more color}
Suppose that $b$ is an integer such that $1<b<n/2$ and $b \neq k$ and $\gcd(b,n) = 1$. If $C_n(1,k)$ is properly colored with $\chi(C_n(1,k))$ colors, and the colors on $v_0$ and $v_b$ are changed to be a new previously unused color, then this coloring is not preserved by any automorphism of $C_n(1,k)$ that carries $E_1$ to $E_k$.
\end{lem}
\begin{proof}
Suppose that $c:V(C_n(1,k)) \to \{1,\dots,\chi(C_n(1,k))+1\}$ is the modified coloring, and let $\phi$ be any automorphism of $C_n(1,k)$ that exchanges $E_1$ and $E_k$. Note that there are two internally vertex-disjoint paths from $v_0$ to $v_b$ along edges of $E_1$, and similarly two such paths from $v_0$ to $v_b$ along edges of $E_k$. Since $\phi$ exchanges $E_1$ and $E_k$, if $\phi$ were to preserve the coloring, then $v_0,v_b$ would either be fixed or mapped to each other under $\phi$, and the paths between them along edges of $E_k$ would be carried to the paths between $v_0,v_b$ using edges of $E_1$. Since these paths along edges in $E_1$ have lengths $b$ and $n-b$, the paths along edges in $E_k$ must have the same lengths. Hence either $bk \equiv b \pmod{n}$ or $bk \equiv n-b \pmod{n}$. Since $\gcd(b,k) = 1$, $b$ has a multiplicative inverse modulo $n$, and these congruences yield $k \equiv 1 \pmod{n}$ or $k \equiv n-1 \pmod{n}$; both statements are contradictions.
\end{proof}

Thus any color-preserving automorphism of the modified coloring in Lemma~\ref{lem: one more color} must act as a dihedral symmetry on the edges of $E_1$. Since $b<n/2$, the only such automorphisms that either fix $v_0$ and $v_b$ or interchange them are the identity automorphism and a single reflection. This will yield a general bound; first we show that such an integer $b$ as in the hypothesis of Lemma~\ref{lem: one more color} always exists for large enough $n$.

\begin{lem}\label{lem: there is a b}
For any $n \geq 13$ and for any integer $k \in \{2,\dots,\lfloor n/2 \rfloor\}$, there exists an integer $b$ such that $1<b<n/2$ and $b \neq k$ and $\gcd(b,n) = 1$.
\end{lem}
\begin{proof}
We will show the stronger statement that when $n \geq 13$, there exist two primes not dividing $n$ in $\{2,\dots,\lfloor n/2 \rfloor\}$. If $k$ were to equal one of these primes, then we could let $b$ be the other one.

It is easy to verify directly that the two primes specified exist for all $n$ satsifying $13 \leq n \leq 22$. Now suppose that $n \geq 23$. If $n$ is not relatively prime to two elements of $\{2,3,5,7,11\}$, then $n \geq 2\cdot 3 \cdot 5 \cdot 7= 210$.

Recall now the result known as Bertrand's Postulate, which states, in one formulation, that whenever $m$ is an integer greater than $3$, then there exists a prime number $p$ with $m < p < 2m$. It follows that there exist prime numbers $p_1,p_2,p_3$ such that $n/16 < p_1 < n/8$ and $n/8 < p_2 < n/4$ and $n/4 < p_3 < n/2$. If $n$ is not relatively prime to at least two of $p_1,p_2,p_3$, then $n \geq p_1p_2 > n^2/128$ and hence $n < 128$, a contradiction to our earlier bound on $n$.
\end{proof}

\begin{thm} \label{thm: non-palindrome}
Given integers $k,n$ such that $1 < k< n/2$ and a proper coloring coloring $c:V(C_n(1,k)) \to$ $\{1,\dots,\chi(C_n(1,k))\}$, let $b$ be an integer such that $1 < b < n/2$ and $b \neq k$ and $\gcd(n,b)=1$. If either 

\[c(v_1),c(v_2),\dots,c(v_{b-1}) \quad \text{or} \quad c(v_{b+1}),c(v_{b+2}),\dots,c(v_{n-1})\] 

is not a palindrome, then $\chi_D(C_{n}(1,k)) \leq 
\chi(C_n(1,k))+1$.
\end{thm}
\begin{proof}
In light of Lemma~\ref{lem: one more color} and the discussion following it, since one of  $c(v_1),c(v_2),\dots,c(v_{b-1})$ and $c(v_{b+1}),c(v_{b+2}),\dots,c(v_{n-1})$ is not a palindrome, recoloring $v_0$ and $v_b$ with a single new color yields a proper coloring where no reflection preserves the coloring; the only color-preserving automorphism of $C_n(1,k)$ is the identity. Hence $\chi_D(C_{n}(1,k)) \leq 
\chi(C_n(1,k))+1$.
\end{proof}

In certain cases, Theorem~\ref{thm: non-palindrome} quickly yields an optimal distinguishing coloring.

\begin{thm} \label{thm: n even k odd}
Given positive integers $k,n$ such that $1 < k <n/2-1$, if $n$ is even and $k$ is odd and $(n,k) \neq (10,3)$, then $\chi_D(C_{n}(1,k)) = 3$.
\end{thm}
\begin{proof}
Note that $C_n(1,k)$ is bipartite, though a proper 2-coloring of $C_n(1,k)$ admits a nontrivial color-preserving rotation, so $\chi_D(C_n(1,k)) \geq 3$. To prove the corresponding upper bound, observe first that no value of $k$ satisfies the hypotheses for any even $n$ less than $10$. When $n \in \{10,12\}$, only $k=3$ satisfies the given inequalities, though we are given that $(n,k) \neq (10,3)$. 

Assume that either $(n,k)=(12,3)$ or $n \geq 13$. Using $b=5$ in the first case and Lemma~\ref{lem: there is a b} in the latter, there is an integer $b$ such that $1<b<n/2$ and $b \neq k$ and $\gcd(b,n) = 1$. Let $c$ be a proper 2-coloring of $C_n(1,k)$; here the vertices $v_i$ with even subscripts recieve one color, and the vertices with odd subscripts receive the other color. Since $n$ is even, $b$ must be odd, and hence $c(v_1),\cdots,c(v_{b-1})$ is not a palindrome, since $c(v_1) \neq c(v_{b-1})$. By Theorem~\ref{thm: non-palindrome}, there is an optimal distinguishing coloring of $C_n(1,k)$ using 3 colors.
\end{proof}

Given the exceptionality of the case $(n,k)=(10,3)$ in Theorems~\ref{thm: E_1 and E_k permutations} and~\ref{thm: n even k odd}, we determine $\chi_D(C_{10}(1,3))$ next. Here the distinguishing chromatic number is quite a bit higher than the chromatic number.

\begin{prop}\label{thm: n=10, k=3}
$\chi_D(C_{10}(1,3)) = 5$.
\end{prop}
\begin{proof}
The graph $C_{10}(1,3)$ is bipartite and may be obtained by deleting the edges $v_i v_{i+5}$ from the complete biparite graph having partite sets $A=\{v_0,v_2,v_4,v_6,v_8\}$ and $B=\{v_1,v_3,v_5,v_7,v_9\}$. Note that if some proper coloring of the vertices assigns the same color to both $v_i,v_j$ in $A$ and the same color (which must different from the first) to $v_{i+5},v_{j+5}$ in $B$, then the involution of $V(C_{10}(1,3))$ written in cycle notation as $(v_i v_j)(v_{i+5} v_{j+5})$ is a color-preserving automorphism of the graph, so such a coloring is not distinguishing.

Let $c:V(C_{10}(1,3)) \to \{1,\dots,\ell\}$ be a distinguishing proper coloring, and suppose by way of contradiction that $\ell < 5$. By a pigeonhole principle, some color must appear on at least three vertices of the graph. Since $c$ is a proper coloring, these three vertices must appear in the same partite set; assume that it is $A$. By the symmetries in $C_{10}(1,3)$, we may assume that these vertices are $v_0,v_2,v_4$, and the color assigned is 1. As explained above, the colors on $v_5,v_7,v_9$ must then be distinct elements of $\{2,\dots,\ell\}$, which forces $\ell = 4$. Since $c$ is a proper coloring, $v_6$ and $v_8$ must also receive color 1, but then when we consider $v_1,v_3$, we see that the pigeonhole principle forces some color from $\{2,\dots,\ell\}$ to appear at least twice on vertices in $B$, and as above we find a color-preserving involution of the vertices of $C_{10}(1,3)$, a contradiction. Thus a distinguishing proper coloring of $C_{10}(1,3)$ requires at least 5 colors, and one can verify that the following map $c:\{v_0,\dots,v_9\} \to \{1,2,3,4,5\}$ provides one.

\begin{center}
\begin{tabular}{rccccc}
$A$: & $c(v_0) = 1$, & $c(v_2) = 2$, & $c(v_4) = 2$, & $c(v_6) = 3$, & $c(v_8) = 3$, \\
$B$: & $c(v_5) = 1$, & $c(v_7) = 4$, & $c(v_9) = 5$, & $c(v_1) = 4$, & $c(v_3) = 5$.
\end{tabular}
\end{center}
\end{proof}

\section{Dihedral symmetries} \label{sec: dihedral symmetries}
In this section we restrict our attention to graphs $C_n(1,k)$ for which $\Aut(C_n(1,k))$ is the dihedral group $D_n$. We will show that often the general bound in the conclusion of Theorem~\ref{thm: non-palindrome} is not optimal, since we may find a distinguishing proper coloring using $\chi(C_n(1,k))$ colors.

When $\Aut(C_n(1,k)) \cong D_n$, every automorphism of $C_n(1,k)$ permutes the edges in $E_1$, to use the notation from Section~\ref{sec: automorphism groups}, and likewise permutes the edges of $E_k$. We may also use more intuitive language, imagining that $C_n(1,k)$ is drawn with its vertices placed, in order of their subscripts, at the vertices of a regular $n$-gon, with the edges of $E_1$ drawn as the sides and the edges of $E_k$ drawn as diagonals of this polygon. This allows us to freely speak of the elements of $\Aut(C_n(1,k)$ as rotations and reflections and to determine the forms of colorings that are preserved under these automorphisms. We do this in Section~\ref{subsec: colorings preserved} below before proceeding in later subsections by the values or parities of $n$ and $k$ (recalling that the case where $n$ is even and $k$ is odd was concluded in Section~\ref{sec: chi plus 1}).

\subsection{Colorings preserved by rotations and reflections} \label{subsec: colorings preserved}
We consider first reflections.

\begin{lem} \label{lem: no reflections}
If a reflection symmetry in $\Aut(C_n(1,k))$ is color-preserving for a given proper coloring of $C_n(1,k)$, then $n$ is even and $k$ is odd.
\end{lem}
\begin{proof}
Picture a drawing of $C_n(1,k)$ as a regular polygon with chords, with the polygon vertices drawn a circle. Every reflection symmetry in $\Aut(C_n(1,k))$ has a corresponding axis of reflection that passes through the center of the circle.
If the axis of reflection passes through the midpoint of a 1-edge, then endpoints of that edge have different colors (since $C_n(1,k)$ is properly colored), and the reflection is not color-preserving. Hence the only possible color-preserving reflection symmetry is one where $n$ is even and the symmetry fixes two ``opposite'' vertices $v_a$ and $v_{a+n/2}$. Here $k$ must be odd, since otherwise the vertices $v_{a-k/2}$ and $v_{a+k/2}$ would have the same color (by the symmetry) but be adjacent, a contradiction.  
\end{proof}

Since the case when $n$ is even and $k$ is odd was handled in Section~\ref{sec: chi plus 1}, we note that for the rest of Section~\ref{sec: dihedral symmetries}, we may ignore reflections when checking for color-preserving symmetries.

The next result deals with rotations in $\Aut(C_n(1,k))$.

\begin{lem} \label{lem: rotations}
If a rotation symmetry in $\Aut(C_n(1,k))$ is color-preserving for a given proper coloring $c$ of $C_n(1,k)$, then the sequence $c(v_0),\dots,v(v_{n-1})$ is periodic with a period that is a proper divisor of $n$.
\end{lem}
\begin{proof}
Suppose that $c:V(G) \to \{1,\dots,\chi(C_n(1,k))\}$ is a proper coloring, and that $\rho$ is a a non-identity rotation in $\Aut(C_n(1,k))$ that preserves the coloring $c$. If $\rho(v_0) = v_r$, then clearly $c(v_i) = c(v_{i+tr})$ for all $t$. This shows that $c(v_0),\dots,c(v_{n-1})$ is periodic.  

In fact, an elementary result from number theory implies that $c(v_0)$ appears on all vertices $v_j$ where $j$ is a multiple of the greatest common divisor of $r$ and $n$. Let $d=\gcd(n,r)$. Now by symmetry $c(v_i) = c(v_{i+td})$ for all $i$ and $t$, showing that the period of $c(v_0),\dots,c(v_{n-1})$ divides $d$ and hence $n$. Since $d<n$, the period is a proper divisor of $n$.
\end{proof}

In light of Lemma~\ref{lem: rotations}, for the rest of Section~\ref{sec: dihedral symmetries}, in verifying that a coloring of $C_n(1,k)$ is preserved by no non-identity symmetry, we need only check that the coloring is not preserved by any rotation $v_i \mapsto v_{i+d}$ where $d$ is a proper divisor of $n$. This allows us a quick result.

\begin{cor}
If $n$ is an odd prime and $\Aut(C_n(1,k)) \cong D_n$, then $\chi_D(C_n(1,k))=\chi(C_n(1,k))$.
\end{cor}
\begin{proof}
By Lemmas~\ref{lem: no reflections} and~\ref{lem: rotations}, any proper coloring of $C_n(1,k)$ is preserved only by the identity in $\Aut(C_n(1,k))$.
\end{proof}

\subsection{Case: $k=2$ or $k=(n-1)/2$}
Theorem~\ref{thm: chromatic num} shows that $\chi(C_n(1,2))$ and $\chi(C_n(1,(n-1)/2)$ are $4$ except when $n=5$ or when $n$ is a multiple of 3 (in the latter case, the chromatic number is $3$). We are able to give optimal distinguishing proper colorings of these graphs. Note first of all that by Corollary~\ref{cor: iso}, $C_n(1,(n-1)/2)$ is isomorphic to $C_n(1,2)$, so it suffices to restrict our attention to $C_n(1,2)$. We may also assume that $n \geq 7$, since distinguishing colorings have already been described in earlier sections for the cases $3 \leq n \leq 6$. 

\begin{thm}\label{thm: k= 2}
For all $n \geq 7$, $\chi_D(C_n(1,2)) = 4$.
\end{thm}
\begin{proof}
By Theorem~\ref{thm: chromatic num}, $\chi_D(C_n(1,2)) \geq \chi(C_n(1,2)) = 4$ if $n$ is not a multiple of 3. If $n$ is a multiple of 3, then the only partition of the vertices of $C_n(1,2)$ into three independent sets is given by grouping the vertices $v_i$ by the congruence class modulo 3 of their subscripts; hence any proper 3-coloring is preserved by the rotation given by $v_i \mapsto v_{i+3}$, and as before we must have $\chi_D(C_n(1,2)) \geq 4$.

By Corollary~\ref{cor: dihedral group}, $\Aut(C_n(1,2))$ is isomorphic to the dihedral group of order $2n$. By Lemma~\ref{lem: no reflections}, a proper coloring of $C_n(1,2)$ will be distinguishing if and only if no color-preserving rotation symmetry exists other than the identity. If $n$ is congruent to 0 or 1 modulo 3, we obtain a proper coloring by assigning $v_0$ the color 4 and greedily coloring $v_1,v_2,\dots,v_{n-1}$ in order with the lowest available color from $\{1,2,3\}$. If $n \equiv 2 \pmod{3}$, we color $C_n(1,2)$ by assigning color 4 to vertices $v_0$ and $v_{n-3}$ and greedily coloring the remaining vertices in order of their subscripts with colors from $\{1,2,3\}$ as before. The placement of color 4 allows for no nontrivial rotational symmetry, so these colorings establish that $\chi_D(C_n(1,2))=4$.
\end{proof}


\subsection{Case: $n$ is even and $k$ is even.}

Before presenting our result when $n$ is even and $k$ is even, we establish some conventions that will also be used in later sections. Taking $n$ and $k$ to be fixed, we first define $q$ and $r$ to be the unique integers such that $n=qk+r$, where $0 \leq r < k$. 

Our distinguishing colorings will often be constructed with \emph{blocks} of colors, that is, sequences of colors to be assigned to vertices $v_i$ with consecutive indices. We may also use \emph{block} to refer to the vertices being assigned that sequence of colors. For example, to color the vertices of $C_n(1,k)$ with the block $B = (1,2,3)$ means to alternately color consecutive vertices with 1, 2, and 3, and we may also refer to subsets of three consecutive vertices colored 1, 2, 3 (in that order) as blocks. Our next result gives a more sophisticated example of coloring with blocks.

\begin{thm} \label{thm: n even, k even}
Given integers $k, n$ such that  $2 <  k < n/2 -1,$ $n$ is even and $k$ is even and $\Aut(C_n(1,k)) \cong D_n$, then $\chi_D(C_{n}(1,k))=3.$
\end{thm}

\begin{proof}

We give a proper 3-coloring of the vertices of $C_n(1,k)$ as follows. Consider the following blocks of $k$ terms, where each color is drawn from $\{1,2,3\}$. Here the bounds on $k$ and the assumption that $k$ is even ensure a consistent definition. For convenience hereafter we represent blocks (and later, portions of blocks) by enclosing them in rectangles.

\medskip

\hfil \begin{tabular}{ccccccccccc}
    $B_1$ & $=$ & ($1$, & $2$, & $3$, & $2$, & $3$, & $\cdots$, & $2$, & $3$, & $1$); \\
    $B_2$ & $=$ & ($2$, & $3$, & $1$, & $3$, & $1$, & $\cdots$, & $3$, & $1$, & $2$); \\
    $B_3$ & $=$ & ($3$, & $1$, & $2$, & $1$, & $2$, & $\cdots$, & $1$, & $2$, & $3$).
\end{tabular} \hfil

\medskip

\textsc{Case 1: $n\geq 3k$ so $q\geq 3$}

\medskip 

 Assign colors to $v_1,\dots,v_{(q-1)k}$ by alternating the use of blocks $B_1$ and $B_2$ on successive collections of $k$ consecutively-indexed vertices. Assign colors to $v_{(q-1)k+1},\dots,v_{qk}$ using the block $B_3$. For any remaining $r$ vertices $v_{qk+1},\dots,v_{n-1}, v_0$, begin by assigning color 2 to $v_{qk+1}$. Assign to $v_0$ the color $3$ if $r=2$ and the color $2$ otherwise. Then color any remaining vertices $v_{qk+2},\dots,v_{n-1}$ by alternating the colors $3$ and $1$. The final $r$ vertices' colors thus create a block $B'$ that is a shortened or partial version of the block $B_2$.

We illustrate this coloring in the the figure below, where each row indicates the colors placed on the sets of $k$ consecutively-indexed vertices in $C_n(1,k)$, beginning in the first row with the colors on $v_1,\dots, v_k$, followed in the second row with the colors on $v_{k+1},\dots,v_{2k}$, and so on. In this way the entries surrounding a vertex's color show the colors on neighboring vertices along $1$-edges (these are the immediately following and preceding numbers) and along $k$-edges (these are the vertically aligned numbers in the previous and following rows; for convenience, the initial block $B_1$ is repeated at the end of the figure). In contrast to the collection of rectangles above, though the blocks $B_1,B_2,B_3$ occupy entire rows, here they are shown as split into two rectangles each, respectively containing $r$ and $k-r$ entries. This allows us to see how the colors from the block $B'$ (which is shaded) are aligned with colors from the preceding block $B_3$ and the following block $B_1$.

\[
\begin{array}{c}
\hspace{2.7cm}\overbrace{\hspace{2.9 cm}}^{r \text{ (even)}}\hspace{.5 cm} \overbrace{\hspace{2.2 cm}}^{k-r \text{ (even)}} \\

\begin{array}{rcc}
B_1: & \framebox{1 2 3 2 3 $\cdots$ 2 3 2} & \framebox{3 2 $\cdots$ 2 3 1}\\

B_2: & \framebox{2 3 1 3 1 $\cdots$ 3 1 3} & \framebox{1 3 $\cdots$ 3 1 2}\\

B_1: & \framebox{1 2 3 2 3 $\cdots$ 2 3 2} & \framebox{3 2 $\cdots$ 2 3 1}\\

B_2: & \framebox{2 3 1 3 1 $\cdots$ 3 1 3} & \framebox{1 3 $\cdots$ 3 1 2}\\

\vdots \hspace{0.25cm} & \vdots & \vdots\\

B_3: & \framebox{3 1 2 1 2 $\cdots$ 1 2 1} & \framebox{2 1 $\cdots$ 1 2 3}\\

B', \text{ then } B_1: & \framebox{{\color{red}{2 3 1 3 1 $\cdots$ 3 1 2}}
} & 
{\framebox{1 2 $\cdots$ 2 3 2}}\\

B_1 \text{ concluded:} & \framebox{3 2 3 2 3 $\cdots$ 2 3 1}
&

\end{array}
\end{array}
\]

To see that this is a proper coloring, note that vertices with consecutive indices $i,i+1$ where $0 \leq i <n$ receive distinct colors by the patterns within the blocks and at their ends. We will also see that each vertex $v_i$ is colored differently than $v_{i-k}$ and $v_{i+k}$. This is apparent from the blocks displayed above if $v_i$ belongs to a block of vertices colored with one of $B_1,B_2,B_3$ and its neighbor $v_{i-k}$ or $v_{i+k}$ does as well, with $B'$ not appearing between the two blocks. To finish the argument, we assume that $r>0$ and consider the vertices $v_i$ for $i \in \{(q-1)k+1,\dots,n\}$, showing that none receives the same color as $v_{i+k}$. These vertices are assigned colors using the blocks $B_3$ and $B'$. In the figure above, these colors appear in the shaded rectangle and on the previous row.

Recalling that $B'$ is constructed as a shortened or truncated version of $B_2$, as we compare the first $r$ entries of $B_3$ with those of $B'$, we see that the colors on vertices $v_i,v_{i+k}$ must differ; this is apparent for colors at the beginning or middle of the blocks, and we use the fact that $r$ is even, so the final entry of $B'$ (which is 3 or 2) is sure to align with an entry of 1 in $B_3$. Likewise, as we compare the entries of $B'$ with the last $r$ entries of $B_1$, the first entry of $B'$ (which is 2) aligns with a 3 from $B_1$, since both $k$ and $r$ are even; similarly, no other color in $B'$ aligns vertically with the same color in $B_1$. Finally, comparing the final $k-r$ entries of $B_3$ with the first $k-r$ entries of $B_1$, having $k$ and $r$ be even ensures that the parities of the relevant entries in $B_3$ differ from the parities of the vertically aligned entries from $B_1$.

When $r>2,$ we shall show that if we switch $v_0$ color to 3, we obtain a proper distinguishing coloring. As presently constructed, when $r>2,$ $v_0$ is a vertex colored 2 having each of its neighbors colored 1. Thus, there is no trouble switching its color to 3. Let's go ahead and switch the color of $v_0$ to 3. This moves keeps the coloring proper while making $v_0$ the only vertex colored 3 having each of its neighbors colored 1. 
Therefore, any automorphism that preserves the coloring must fix that vertex. By Lemmma~\ref{lem: no reflections}. This leaves only the trivial automorphism, and the coloring is distinguished.

\bigskip

When $r=2,$ we will show that the vertices $v_{-2},v_{-1},v_{0}$ must be fixed by color-preserving symmetries. Hence, the coloring will be distinguishing, since only the identity rotation preserves it. We depict the coloring in this case as we did before, with $k$ numbers in each row and the block $B'$ appearing as shaded.

\medskip
\[
\begin{array}{c}
\overbrace{\hspace{.6 cm}}^{r=2}\hspace{.3 cm} \overbrace{\hspace{2.6 cm}}^{k-r \text{ (even)}} \\

\begin{array}{cc}
\framebox{1 2} & \framebox{3 2 $\cdots$ 3 2 3 1}\\

\framebox{2 3} & \framebox{1 3 $\cdots$ 1 3 1 2}\\

\framebox{1 2} & \framebox{3 2 $\cdots$ {\color{blue}{3 2 3}}  1}\\

\framebox{2 3} & \framebox{1 3 $\cdots$ 1 3 1 2}\\

\vdots & \vdots\\

\framebox{1 2} & \framebox{3 2 $\cdots$ 3 2 3 1}\\


\framebox{3 1} & \framebox{2 1 $\cdots$ 2 1 2 3}\\

\framebox{{\color{red}{2 3}}} & 
\framebox{1 2 $\cdots$ 3 2 3 2}
\\
\framebox{3 1}
 &
\end{array}
\end{array}
\]

\medskip

The vertex $v_0$ labeled 3, which is the last vertex in the shortened $B'$ block, has 3 neighbors labeled 1 and one neighbor labeled 2.  In looking at $v_{-1}$, this vertex is labeled 2 with all neighbors labeled 3, and in looking at $v_{-2}$, this is a vertex labeled 3 with one neighbor labeled 1 and the rest labeled 2. The only other vertices that have the same neighbors as $v_0$ are the second to last entries in a $B_1$ block surrounded by $B_2$ blocks, call one of these vertices $v_i$. However $v_{i-2}$ is a vertex labeled 3 with at least two 1 neighbors, or in the case where $k=4$ a vertex labeled 1.


\textsc{Case 2: $n=2k+r$ so $q=2$}

\medskip

Now suppose $n = 2k + r, $ where $ 1< r < k.$ Since $n$ is even, then $r$ is also even. Consider the blocks of colors $E_1, E_2$ of length $k$ and $M$ of length $r.$

\medskip

\hfil \begin{tabular} {c}$E_1= (1, 2, 1, 2, \cdots, 1, 2, 1, 2) $ \\ $E_2= (3, 1, 3, 1, \cdots, 3, 1, 3, 1)$ \\ $M= (2, 3, 2, 3, \cdots, 2, 3, 2, 3)$ \end{tabular} \hfil 

\medskip

Color the sets of vertices $\{v_1, v_2, \cdots, v_k\}$ and $\{v_{k+1}, v_{k+2}, \cdots, v_{2k}\}$ using the blocks of colors  $E_1$ and $E_2,$ respectively. Then use $M$ to color the remaining $r$ vertices. Below is the coloring $C$ of $C_{2k+r}(1, k)$ with a detailed representation of these blocks of colors. The light-shaded blocks in rows 3 and 4 correspond exactly to the first $k$ vertices in row 1. 

\medskip

\[
\begin{array}{c}
\overbrace{\hspace{2.2 cm}}^{r}\hspace{.1 cm} \overbrace{\hspace{2.2 cm}}^{k-r} \\

\left[\begin{array}{cc}
\framebox{1 2 $\cdots$ 1 2} & \framebox{1 2 $\cdots$ 1 2} \\

\framebox{3 1 $\cdots$ 3 1} & \framebox{3 1 $\cdots$ {\color{red}{3 1}}} \\

\framebox{{\color{red}{2 3}} $\cdots$ 2 3} & \color{gray}{\framebox{1 2 $\cdots$ 1 2}} \\ \color{gray}{\framebox{1 2 $\cdots$ 1 2}} & 
\end{array}
\right]
\end{array}
\]

\medskip

Since $r$ and $k$ are both even, we can easily see from the detailed representation above that $C$ is a proper coloring. By lemma~\ref{lem: no reflections}, it suffices to show that there is not a nontrivial rotation symmetry that preserves coloring. The string of vertices $v_{2k-1}, v_{2k}, v_{2k+1}, v_{2k+2}$ (as highlighted in the block diagram) is the only string of vertices labeled (3, 1, 2, 3) as every other pair of vertices labeled (1, 2) is preceded or followed by another pair of vertices labeled (1, 2). Hence, $C$ is distinguishing.
\medskip

\end{proof}


\subsection{Case: $n$ is odd and $k$ is odd} 

We shall treat the cases where $2k+1<n <3k$ and $ n \geq 3k$ separately. In either case, we propose a proper 3-coloring that is distinguishing. To do so, we conveniently label the vertices of $G$ consecutively as $v_1 , \cdots , v_n,$ though we shall realize later that the coloring works independently of the vertex label. Note that since both $k, n$ are odd, $v_{-k}= v_{(q-1)k+r}$ is labeled an even number. 

\begin{thm}\label{thm:both odd}
Given integers $k, n$ such that  $2 <  k < (n-1)/2,$ if both $n$ and $k$ are odd and $\Aut(C_n(1,k)) \cong D_n,$ then $\chi_D(C_{n}(1,k))=3.$
\end{thm}

\begin{proof} As indicated, we split it into three cases:

\textsc{Case 1: $ n \geq 3k$}
\medskip

 Let $q, r$ be integers such that $0 \leq r < k$ and $n=qk+r.$ Consider the following 3-proper coloring $C$ of the vertices: assign color 1 to all odd-indexed vertices $v_i \in \{v_1, v_2, \cdots, v_{-k}\}.$ Next, assign color 3 to even-indexed vertices $v_i \in \{v_{k+1}, \cdots, v_n\},$ not including $v_n$ since $n$ is considered odd. Lastly, assign color 2 to all other vertices in $\{v_1, v_2, \cdots, v_k\} \cup \{v_{-k+1},v_{-k+2}, \cdots, v_n\}.$ In other words, for even $i$ such that $1< i< k,$ the vertex $v_i$ is colored 2 and for odd $i$ such that $(q-1)k< i \leq n$ (including $n$), we have $v_i$ colored 2. We claim that the coloring is proper. Moreover, it is distinguishing.

 We proceed to prove that $C$ is proper by showing that the sets made up of vertices with the same color are all independent sets. Consider the set of all 1-colored vertices and denote it $V_1.$ Thus, $V_1= \{v_i | i \text{ is odd and } 1 \leq i \leq (q-1)k+r \}$ by construction. We show that $V_1$ is an independent set. Take $v_i \in V_1.$ Thus, $v_{i-1}$ and $v_{i+1}, v_{i+k}$ (since $k$ is odd) are even-labeled vertices and do not belong to $V_1.$ Moreover, $v_{i-k}$ is  an even-labeled vertex unless $1 \leq i \leq k,$ in which case $v_{i-k} \in \{v_{-k+1},v_{-k+2}, \dots, v_n\}.$ Therefore, $V_1$ is indeed an independent set.

A similar argument can be made for the set $V_3$ of the vertices colored 3. This time, we have $V_3 = \{v_i| i \text{ is even and } k+1 \leq i \leq n\}$ and $v_{i-1}, v_{i+1}, v_{i-k}$ are odd-labeled vertices. Moreover, $v_{i+k}$ is an odd-labeled vertex unless $(q-1)k+r+1 \leq i \leq n,$ in which case $v_1 \in \{ v_1, v_2, \cdots, v_k\}.$ 

Lastly, we show that the set $V_2$ of vertices colored 2 is also an independent set. By construction, we have
 $$V_2= \{v_i|i \text{ is even and } 1\leq i \leq k\} \cup \{v_i|i \text{ is odd and } (q-1)k+r+1 \leq i \leq n\}$$

Here is what $C$ looks like when restricted to $\{v_{-k+1}, v_{-k+2}, \cdots, v_{-1}, v_n, v_1, v_2, \cdots, v_k\}:$

\begin{center} \includegraphics[width=8cm, scale=4]{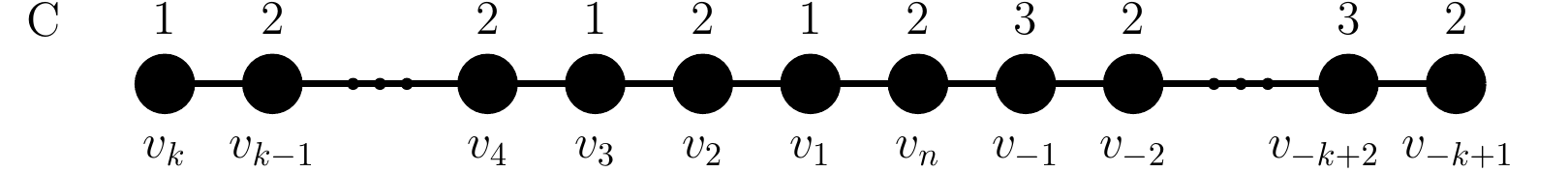}

\end{center}

As we can see from the above illustration, the coloring $C$ restricted to $\{v_{-k+1}, v_{-k+2}, \cdots, v_{-1}, v_n, v_1, v_2, \cdots, v_k\}$ is a sequential-vertex coloring of the vertex list $v_{-k+1}, v_{-k+2}, \cdots, v_{-1}, v_n, v_1, v_2, \cdots, v_k,$ which starts with the alternation of colors 2 and 3 on the first k vertices to end up with color 3 being replaced with color 1 on the last k vertices. Thus, given any pair of vertices $v_i, v_j$ coloured 2, we have $\pm(i-j)= 2\ell,$ where $\ell \in \mathbb{Z}.$ Therefore, the set of vertices colored 2 form an independent set. In summary, $C$ is a proper coloring.

 It remains to show that it is distinguishing. By lemma 7.1, it suffices  to show that the only color-preserving rotation symmetry is the identity. Suppose there is a non-identity rotation symmetry. Thus, by Lemma 7.2, there exists a proper divisor $d$ of $n$ such that $\rho(v_n)= v_d.$ This means that $C(v_n) = C(v_d).$  Since $d$ is a proper divisor of $n,$ thus $d$ must be a number between 1 and $k$ by construction. In particular, $d$ must also be a multiple of 2 since for any pair of vertices $v_i, v_j$ coloured with 2, we have $\pm(i-j)= 2\ell,$ where $\ell \in \mathbb{Z}.$ The fact that $d$ is a proper divisor of odd $n$ and also a multiple of $2$ yields a contradiction. Hence, $C$ is indeed distinguishing.

\textsc{Case 2: If $n= 2k+r,$ where $1 < r < k$ and $2r\leq k-1$}
\medskip

Note that in this case, $r$ is odd and we consider integers $\ell$ and $p$ such that $k=pr+\ell$ and $0<\ell<r$. 

Consider the following arrangement of colors where the first two rows are made up of $p$ blocks of length $r$ with a possible shorter block of length $\ell$ (if $\ell >0$) added to the end. The third row consists of just one block of length $r$. Furthermore, the first row is used to color the sequence of vertices $v_1, v_2, \cdots, v_k,$ second row to color vertices $v_{k+1}, v_{k+2}, \cdots, v_{2k},$ and third row to color the remaining vertices. 

\[
\begin{array}{c}
\overbrace{\hspace{2.2 cm}}^{r}\hspace{.3 cm} \overbrace{\hspace{2.2 cm}}^{r} \hspace{.3cm}\overbrace{\hspace{2.2 cm}}^{r}
\hspace{1 cm}
\overbrace{\hspace{2.2 cm}}^{r} \hspace{.3 cm} \overbrace{\hspace{.5 cm}}^{\ell}\\

\left[\begin{array}{cccccc}
\framebox{\hspace{.5cm} $R_1$ \hspace{.5cm} } & \framebox{3 2 $\cdots$ 3 2 1} & \framebox{3 2 $\cdots$ 3 2 1}&\cdots & \framebox{3 2 $\cdots$ 3 2 1} & \framebox{$L_1$}\\

\framebox{\hspace{.5cm} $R_2$ \hspace{.5cm} } & \framebox{2 1 $\cdots$ 2 1 3} & \framebox{2 1 $\cdots$ 2 1 3}&\cdots & \framebox{2 1 $\cdots$ 2 1 3} & \framebox{$L_2$} \\

\framebox{\hspace{.5cm} $R_3$ \hspace{.5cm} } &  & & & & 
\end{array}
\right]
\end{array}
\]

where the blocks $R_1, R_2, R_3$ of length $r$ and blocks $L_1, L_2$ of length $\ell$ vary with respect to $\ell.$ In particular, we have

\bigskip

\[
\begin{tabular}{|c|c|c|c|c|}
\hline
& $\ell = 0$ & $\ell = 1$ & $\ell >1$ and odd & $\ell>1$ and even \\
\hline
$R_1$ & \framebox{1 3 $\cdots$ 1 3 1} & \framebox{3 2 $\cdots$ 3 2 1} & \framebox{3 2 $\cdots$ 3 2 1}& \framebox{1 3 $\cdots$ 1 3 1}\\
\hline 
$R_2$ & \framebox{2 1 3 $\cdots$ 1 3 } & \framebox{2 1 $\cdots$ 2 1 3} & \framebox{2 1$\cdots$ 2 1 3}& \framebox{3 2$\cdots$ 3 2 3}\\

\hline 
$R_3$ & \framebox{1 3 2  $\cdots$ 3 2} & \framebox{1 2 $\cdots$ 1 2 1} & \framebox{3 2 $\cdots$ 3 2 1}& \framebox{2 1  $\cdots$ 2 1 2}\\
\hline
$L_1$ &N/A&  \framebox{3} & \framebox{2 1 $\cdots$ 2 1 3} & \framebox{3 1 $\cdots$ 3 1}\\
\hline
$L_2$ & N/A & \framebox{2} & \framebox{1 3 $\cdots$ 1 3 1} & \framebox{2 3 $\cdots$ 2 3}\\
\hline
\end{tabular}
\]

\bigskip

We proceed to show that the coloring is proper and distinguishing on a case-by-case basis. We start with the case where $\ell =0.$

\textsc{Subcase 1: Let $\ell = 0$ }.

Thus, $k= pr.$ Let $L_1$ and $L_2$ be empty and the other blocks $R_1, R_2,$ and $R_3$ as given below. Note that the added light-shaded blocks in rows 3 and 4 represent row 1 re-positioned so that the first $k-r$ vertices (positions) in row 1 are vertically aligned with the last $k-r$ vertices in row 2 and the last $r$ vertices (positions) in row 1 are vertically aligned with the $r$ vertices in row 3.

\[
\begin{array}{c}
\overbrace{\hspace{2.4 cm}}^{r}\hspace{.4 cm} \overbrace{\hspace{2.4 cm}}^{r} \hspace{.35cm}\overbrace{\hspace{2.5 cm}}^{r}
\hspace{1cm}
\overbrace{\hspace{2.5 cm}}^{r}\\

\left[\begin{array}{ccccc}
\framebox{1 3 1 $\cdots$ 1 3 1} & \framebox{3 2 3 $\cdots$ 3 2 1} & \framebox{3 2 3 $\cdots$ 3 2 1}& \cdots & \framebox{3 2 3 $\cdots$ 3 2 1}\\

\framebox{{\color{red}{2}} 1 3 $\cdots$ 3 1 3} & \framebox{2 1 2 $\cdots$ 2 1 3} & \framebox{2 1  2 $\cdots$ 2 1 3}& \cdots & \framebox{2 1 2 $\cdots$ 2 1 3} \\

\framebox{1 3 2 $\cdots$ 2 3 2} & \color{gray}{\framebox{1 3 1 $\cdots$ 1 3 1}} & \color{gray}{\framebox{3 2 3 $\cdots$ 3 2 1}} & \color{gray}{\cdots} & 
 \color{gray}{\framebox{3 2 3 $\cdots$ 3 2 1}} \\ \color{gray}{\framebox{3 2 3 $\cdots$ 3 2 1}} & & &
\end{array}
\right]
\end{array}
\]

We can easily see from the above detailed arrangement of colors that the coloring is proper. By Lemma~\ref{lem: no reflections}, it suffices to show that there is not a nontrivial rotation symmetry that preserves the coloring. To this end, note that the vertex $v_{k+1}$ is the only vertex labeled 2 whose neighbors are all labeled 1. Hence, the coloring is distinguishing with respect to the dihedral group.

\bigskip

\textsc{Subcase 2: Let $\ell = 1$ (thus, $p$ is even). }

Let the blocks $R_1, R_2, R_3$ of length $r$ and blocks $L_1, L_2$ of length 1 be as given below.

\[
\begin{array}{c}
\overbrace{\hspace{2.6 cm}}^{r}\hspace{.4 cm} \overbrace{\hspace{2.4 cm}}^{r} \hspace{.35cm}\overbrace{\hspace{2.5 cm}}^{r}
\hspace{1cm}
\overbrace{\hspace{2.5 cm}}^{r} \hspace{.2cm} \overbrace{\hspace{.5 cm}}^{\ell}\\

\left[\begin{array}{cccccc}
\framebox{3 2 3 $\cdots$ 3 2 \, 1} & \framebox{3 2 3 $\cdots$ 3 2 1} & \framebox{3 2 3 $\cdots$ 3 2 1}& \cdots & \framebox{3 2 3 $\cdots$ 3 2 1}& \framebox{3}\\

\framebox{2 1 2 $\cdots$ 2 1 \, 3} & \framebox{2 1 2 $\cdots$ 2 1 3} & \framebox{2 1  2 $\cdots$ 2 1 3}& \cdots & \framebox{2 1 2 $\cdots$ 2 1 3}& \framebox{2}\\

\framebox{1 2 1 $\cdots$ 1 {\color{red}{2}} \, 1} & \color{gray}{\framebox{3 2 3 $\cdots$ 3 2 1}} & \color{gray}{\framebox{3 2 3 $\cdots$ 3 2 1}} & \color{gray}{\cdots} & \color{gray}{\framebox{3 2 3 $\cdots$ 3 2 1}} & \color{gray}{\framebox{3}} \\ 

\color{gray}{\framebox{2 3  2 $\cdots$ 2 1 $\mid$ 3}} & & &
\end{array}
\right]
\end{array}
\]

Here, $v_{-1}$ is the only vertex labeled 2 whose neighbors are all labeled 1. Similarly, the coloring is proper and distinguishing. 

\bigskip

\textsc{Subcase 3: Let $\ell > 1$ and odd (thus, $p$ is even).}

Let the blocks $R_1, R_2, R_3$ and blocks $L_1, L_2$ be as given below.

\[
\begin{array}{c}
\overbrace{\hspace{3.9cm}}^{r}\hspace{.4 cm} \overbrace{\hspace{2.3 cm}}^{r} \hspace{.4cm}\overbrace{\hspace{2.5 cm}}^{r}
\hspace{.9cm}
\overbrace{\hspace{2.7 cm}}^{r} \hspace{.3cm} \overbrace{\hspace{2.6 cm}}^{\ell}\\

\left[\begin{array}{cccccc}
\framebox{3 2 $\cdots$ 3 2 \, 3 2  $\cdots$ 3 2 1} & \framebox{3 2 3 $\cdots$ 3 2 1} & \framebox{3 2 3 $\cdots$ 3 2 1}& \cdots & \framebox{3 2 3 $\cdots$ 3 2 1}& \framebox{2 1 2 $\cdots$ 2 1 3}\\

\framebox{2 1 $\cdots$ 2 1 \, 2 1 $\cdots$ 2 1 3} & \framebox{2 1 2 $\cdots$ 2 1 3} & \framebox{2 1 2 $\cdots$ 2 1 3}& \cdots & \framebox{2 1 2 $\cdots$ 2 {\color{red}{1 3}}}& \framebox{\color{red}{1 3 1 $\cdots$ 1 3 1}}\\

\framebox{3 2 $\cdots$ 3 2 \, 3 2 $\cdots$ 3 2 1} & \color{gray}{\framebox{3 2 3 $\cdots$ 3 2 1}} & \color{gray}{\framebox{3 2 3 $\cdots$ 3 2 1}} & \color{gray}{\cdots} & \color{gray}{\framebox{3 2 3 $\cdots$ 3 2 1}} & \color{gray}{\framebox{3 2 3 $\cdots$ 3 2 3}} \\ 

\color{gray}{\framebox{2 3 $\cdots$ 2 1 $\mid$ 2 1 $\cdots$ 2 1 3}} & & & \\
\underbrace{\hspace{1.9cm}}_{\ell - r\text{ (even)}}\hspace{.1 cm} \underbrace{\hspace{2 cm}}_{\ell\text{ (odd)}} &&&&

\end{array}
\right]
\end{array}
\]

When restricted to the sequence of vertices $v_{2k-\ell-2}, v_{2k-\ell-1}, \cdots, v_{2k},$ the above proper coloring alternates labels 1 and 3. Since $\ell > 1,$ the length of this sequence is at least five, which makes it the longest of its kind. As a result, no nontrivial rotation preserves the coloring. Therefore, by Lemma~\ref{lem: no reflections}, the coloring is distinguishing with respect to the dihedral group.

\bigskip

\textsc{Subcase 4: Let $\ell > 1$ and even (thus $p$ is odd).}

Let the blocks $R_1, R_2, R_3$ and blocks $L_1, L_2$ be as given below.

\[
\begin{array}{c}
\overbrace{\hspace{4 cm}}^{r}\hspace{.4 cm} \overbrace{\hspace{2.4 cm}}^{r} \hspace{.35cm}\overbrace{\hspace{2.5 cm}}^{r}
\hspace{1cm}
\overbrace{\hspace{2.7 cm}}^{r} \hspace{.2cm} \overbrace{\hspace{2.4 cm}}^{\ell}\\

\left[\begin{array}{cccccc}
\framebox{1 3 $\cdots$ 1 3 1 \, 3 $\cdots$ 1 3 1} & \framebox{3 2 3 $\cdots$ 3 2 1} & \framebox{3 2 3 $\cdots$ 3 2 1}& \cdots & \framebox{3 2 3 $\cdots$ 3 2 1}& \framebox{3 1 3 $\cdots$ 1 3 {\color{red}{1}}}\\

\framebox{{\color{red}{3}} 2  $\cdots$ 3 2 3 \, 2 $\cdots$ 3 2 3} & \framebox{2 1 2 $\cdots$ 2 1 3} & \framebox{2 1  2 $\cdots$ 2 1 3}& \cdots & \framebox{2 1 2 $\cdots$ 2 1 3}& \framebox{2 3 2 $\cdots$ 3 2 3}\\

\framebox{2 1 $\cdots$ 2 1 2 \, 1 $\cdots$ 2 1 2} & \color{gray}{\framebox{1 3 1 $\cdots$ 1 3 1}} & \color{gray}{\framebox{3 2 3 $\cdots$ 3 2 1}} & \color{gray}{\cdots} & \color{gray}{\framebox{3 2 3 $\cdots$ 3 2 1}} & \color{gray}{\framebox{3 2 3 $\cdots$ 2 3 2}} \\ 

\color{gray}{\framebox{3 2 $\cdots$ 3 2 1 $\mid$ 3 $\cdots$ 1 3 1}} & & & \\
\underbrace{\hspace{1.9cm}}_{\ell - r\text{ (odd)}}\hspace{.1 cm} \underbrace{\hspace{2 cm}}_{\ell\text{ (even)}} &&&&

\end{array}
\right]
\end{array}
\]

By construction, the coloring is proper. Moreover, $v_r$ is the only vertex labeled 1 that has all its neighbors labeled 3. Also, $v_{k+1}$ is the only vertex labeled 3 with exactly two neighbors labeled 2. 
As a result, no nontrivial rotation preserves the coloring. Therefore, by Lemma~\ref{lem: no reflections}, the coloring is distinguishing with respect to the dihedral group.

\bigskip

\textsc{Case 3: If $n=2k+r$, where $1 < r< k$ and $2r\geq k+1$}
\medskip

\textsc{Subcase 1: Let $2r-k < k-r$ with $2r-k>1$}.

Consider the following arrangements of sequences of colors below, where the first two rows represents sequences of length $k$ and the third row is a sequence of length $r.$ The remaining light-shaded blocks in rows 3 and 4 represent row 1 re-positioned so that the first $k-r$ vertices (positions) in row 1 are vertically aligned with the last $k-r$ vertices (positions) in row 2 and the last $r$ vertices (positions) in row 1 are vertically aligned with the $r$ vertices in row 3. 

\[
\begin{array}{c}
\overbrace{\overbrace{\hspace{3 cm}}^{2r-k \text{ (odd)}}\hspace{.3 cm} 
\overbrace{\hspace{2.5 cm}}^{2k-3r \text{ (odd)}}}^{k-r\text{ (even)}}\hspace{.3 cm}
\overbrace{\hspace{3 cm}}^{2r-k \text{ (odd)}}
\hspace{.3 cm}
\overbrace{\overbrace{\hspace{2.5 cm}}^{2k-3r \text{ (odd)}}
\hspace{.3 cm}
\overbrace{\hspace{3 cm}}^{2r-k \text{ (odd)}}}^{k-r\text{ (even)}}\\

\left[\begin{array}{ccccc}
\framebox{2 1 2 1 2 $\cdots$ 2 1 2} & \framebox{1 2 1 2 $\cdots$ 2 1} & \framebox{3 1 3 1 3 $\cdots$ 3 1 3} & \framebox{2 1 2 1 $\cdots$ 1 2} &\framebox{1 2 1 2 1 $\cdots$ 1 2 1}\\

\framebox{3 2 1 2 1 $\cdots$ 1 2  1} & \framebox{2 1 2 1 $\cdots$ 1 2} & \framebox{1 3 1 3 1 $\cdots$ 1 3 1} & \framebox{3 2 1 2 $\cdots$ 2 1} & \framebox{2 1 2 1 2 $\cdots$ 2 1 2} \\

\framebox{1 {\color{red}{3 2 3 2 $\cdots$ 2 3 2}}} & \framebox{{\color{red}{3 2 3 2 $\cdots$ 2 3}}} & \framebox{{\color{red}{2}} 1 3 1 3 $\cdots$ 3 1 3} &  \color{gray}{\framebox{2 1 2 1 $\cdots$ 1 2}} & \color{gray}{\framebox{1 2 1 2 1 $\cdots$ 1 2 1} }\\

\color{gray}{\framebox{3 1 3 1 3 $\cdots$ 3 1 3} } & \color{gray}{\framebox{2 1 2 1 $\cdots$ 1 2} } & \color{gray}{\framebox{1 2 1 2 1 $\cdots$ 1 2 1} } &
\end{array}
\right]
\\

\hspace{10 cm}\underbrace{\hspace{5.8 cm}}_{k-r\text{ (even)}}
\end{array}
\]

Use row 1 to color the sequence of vertices $v_1, v_2,\cdots, v_k,$ row 2 to color the sequence of vertices $v_{k+1}, v_{k+2}, \cdots, v_{2k},$ and the first three blocks in row 3 to color the remaining consecutively-indexed vertices. It is not hard to see from the above detailed arrangement of colors that the coloring is proper. By Lemma~\ref{lem: no reflections}, it suffices to show that there is not a nontrivial rotation symmetry that preserves the coloring. Note that the highlighted vertices above (from $v_{2k+2}$ to $v_{-2r+k +1}$) is a unique string of vertices labeled alternately 2 and 3. The longest such string outside of these vertices is only 2 vertices long, and therefore this string is unique as $2r-k > 1$. Therefore, no rotational symmetry will fix the coloring and thus this coloring is distinguishing. 

\textsc{Subcase 2: Let $2r-k < k-r$ with $2r-k=1$ and $2k-3r >1$}.

In this case, the coloring must be altered slightly because of the construction of the blocks in the previous subcase. Consider the coloring below in which the colors of the last two vertices are changed from the last subcase.  

\[
\begin{array}{c}
\overbrace{\overbrace{\hspace{.2 cm}}^{ 2r-k}\hspace{.3 cm} 
\overbrace{\hspace{2.2 cm}}^{2k-3r \text{ (odd)}}}^{k-r\text{ (even)}}
\hspace{.3 cm}
\overbrace{\hspace{.2 cm}}^{2r-k}
\hspace{.3 cm}
\overbrace{\overbrace{\hspace{2.2 cm}}^{2k-3r \text{ (odd)}}
\hspace{.3 cm}
\overbrace{\hspace{.2 cm}}^{2r-k}}^{k-r\text{ (even)}}\\

\left[\begin{array}{ccccc}
\framebox{2} & \framebox{1 2 1 2 $\cdots$ 2 1} & \framebox{3} & \framebox{2 1 2 1 $\cdots$ 1 2} &\framebox{1}\\

\framebox{{\color{red}{3}}} & \framebox{2 1 2 1 $\cdots$ 1 2} & \framebox{1} & \framebox{3 2 1 2 $\cdots$ 2 1} & \framebox{2} \\

\framebox{1} & \framebox{3 2 3 2 $\cdots$ 2 1} & \framebox{3} &  \color{gray}{\framebox{2 1 2 1 $\cdots$ 1 2}} & \color{gray}{\framebox{1} }\\

\color{gray}{\framebox{3} } & \color{gray}{\framebox{2 1 2 1 $\cdots$ 1 2} } & \color{gray}{\framebox{1} } &
\end{array}
\right]
\\

\hspace{4.5 cm}\underbrace{\hspace{3.3 cm}}_{k-r\text{ (even)}}
\end{array}
\]

The coloring given above is proper and note that the highlighted vertex, $v_{k+1}$, is the only vertex labeled 3 with two of its neighbors labeled 2 and two labeled 1. Because this vertex is unique, no rotational symmetry can fix the coloring and thus the coloring is distinguished.

\textsc{Subcase 3: Let $2r-k = 1$ and $2k-3r = 1$ }

Note that if both $2r-k = 1$ and $2k-3r = 1$ we have the specific case of $C_{13}(1,5)$ which has a chromatic distinguishing number of 4 and is investigated in Section 8.

\textsc{Subcase 4: Let $2r-k > k-r$ }

Once again, consider the following arrangement of sequences of colors.

\[
\begin{array}{c}
\overbrace{\hspace{3 cm}}^{k-r \text{ (even)}}\hspace{.4 cm} \overbrace{\hspace{3.2 cm}}^{2r-k \text{ (odd)}}
\hspace{.4 cm}
\overbrace{\hspace{3 cm}}^{k-r \text{ (even)}}\\

\left[\begin{array}{ccc}
\framebox{2 1 2 1 2 $\cdots$ 1 2 1} & \framebox{3 2 $\cdots$ 3 \, 2 3 $\cdots$ 2 3} & \framebox{2 1 2 1 2 $\cdots$ 1 2 1}\\

\framebox{{\color{red}{3}} 2 1 2 1 $\cdots$ 2 1 2} & \framebox{1 3 $\cdots$ 1 \, 3 1 $\cdots$ 3 1} & \framebox{3 2 1 2 1 $\cdots$ 2 1 2} \\

\framebox{1 3 2 3 2 $\cdots$ 3 2 3} & \framebox{2 1 $\cdots$ 2 \, 1 2 $\cdots$ 1 3} & \color{gray}{\framebox{2 1 2 1 2 $\cdots$ 1 2 1} } \\ 
\color{gray}{\framebox{3 2 3 2 3 $\cdots$ 2 3 2}}& \color{gray}{\framebox{3 2 $\cdots$ 3 $\mid$ 2 1 $\cdots$ 2 1} } & \\

 & \underbrace{\hspace{1.6cm}}_{3r-2k\text{ (odd)}}\hspace{.09 cm} \underbrace{\hspace{1.7 cm}}_{k-r\text{ (even)}} &
\end{array}
\right]
\end{array}
\]

Proceed similarly to color the vertices and observe that the coloring is proper. Moreover, the vertex $v_{k+1}$ is the only vertex labeled 3 with exactly two of its neighbors labeled 2 and the others labeled 1. Hence, the coloring is distinguishing with respect to the dihedral group.
\bigskip

\end{proof}

\subsection{Case: $n$ is odd and $k$ is even} 

We will proceed by considering the cases where $k\geq n/3$ and $n/3 < k < n/2-1$. Furthermore, we will need to consider within the case that $n/3 < k < n/2-1$, when $r\leq k/2$ and $r>k/2$. 

\begin{thm}\label{thm:n odd and k even}
Given integers $k, n$ such that  $2 <  k < (n-1)/2,$ if $n$ is odd and $k$ is even and $\Aut(C_n(1,k)) \cong D_n,$ then $\chi_D(C_{n}(1,k))=3.$
\end{thm}

\begin{proof}

\textsc{Case 1: Let $n/3<k<n/2 -1$ such that $n=2k+r$ hence $r$ odd and $ r\leq k/2$}.

\bigskip 

Let $r\leq k/2$ so that $k=mr+\ell$ for some integers $\ell<r$ and $m \geq 2$.

Consider the following arrangement of colors where the first two rows are made up of blocks of length $r$ with a possible shorter block of length $\ell$ (if $\ell>0$) added to the end. The third row consists of just one block of length $r$. Note that in the fourth row, the block is comprised of a portion of block $B$ of length $r-\ell$ (we will call $B'$) and $E_1$ of length $\ell$.

\[
\begin{array}{c}
\overbrace{\hspace{1.4 cm}}^{r}\hspace{.2 cm} \overbrace{\hspace{.5 cm}}^{r}
\hspace{1 cm} \overbrace{\hspace{.4 cm}}^{r}
\hspace{.2 cm}
\overbrace{\hspace{.4 cm}}^{\ell}\\

\left[
\begin{array}{ccccc}
\framebox{\,\,\,\,\, A \,\,\,\,\,} & \framebox{B} & \cdots & \framebox{B} & \framebox{$E_1$}\\

\framebox{\,\,\,\,\, C \,\,\,\,\,} & \framebox{D} & \cdots & \framebox{D} & \framebox{$E_2$} \\

\framebox{\,\,\,\,\, E \,\,\,\,\,} & \color{gray}{\framebox{A}}  & \color{gray}{\cdots} & 
 \color{gray}{\framebox{B}} & \color{gray}{\framebox{B}}\\

 \color{gray}{\framebox{B' $\mid$ $E_1$ }} & & & &
\end{array}
\right]
\end{array}
\]

\bigskip

\textsc{Subcase 1: Let $\ell = 0$ }.

\bigskip
Let $E_1$ and $E_2$ be empty and the other blocks as given below. 

\[
\begin{array}{c}
\overbrace{\hspace{2.4 cm}}^{r}\hspace{.3 cm} \overbrace{\hspace{2.4 cm}}^{r}
\hspace{1cm}
\overbrace{\hspace{2.4 cm}}^{r}\\
\left[\begin{array}{cccc}
\framebox{1 3 1 $\cdots$ 1 3 1} & \framebox{3 2 3 $\cdots$ 3 2 1} & \cdots & \framebox{3 2 3 $\cdots$ 3 2 1}\\
\framebox{{\color{red}{2}} 1 3 $\cdots$ 3 1 3} & \framebox{2 1 2 $\cdots$ 2 1 3} & \cdots & \framebox{2 1 2 $\cdots$ 2 1 3} \\
\framebox{1 3 2 $\cdots$ 2 3 2} & \color{gray}{\framebox{1 3 1 $\cdots$ 1 3 1}}  & \color{gray}{\cdots} & 
 \color{gray}{\framebox{3 2 3 $\cdots$ 3 2 1}}\\
\color{gray}{\framebox{3 2 3 $\cdots$ 3 2 1}} & & &
\end{array}
\right]
\end{array}
\]

\medskip

The vertex $v_{k+1}$ highlighted above is the only vertex labeled 2 whose neighbors are all labeled 1. Because this vertex is unique, the coloring is distinguished by dihedral symmetries. 

\bigskip

\textsc{Subcase 2: Let $\ell = 1$ }.

\bigskip

Let the blocks be given as below with $E_1$ and $E_2$ of length 1. 

\[
\begin{array}{c}
\overbrace{\hspace{2.4 cm}}^{r}\hspace{.3 cm} \overbrace{\hspace{2.4 cm}}^{r}
\hspace{1.2 cm}
\overbrace{\hspace{2.4 cm}}^{r} \hspace{.3 cm} \overbrace{\hspace{.2 cm}}^{\ell} \\

\left[\begin{array}{ccccc}
\framebox{3 2 3 $\cdots$ 3 2 \, 1} & \framebox{3 2 3 $\cdots$ 3 2 1} & \cdots & \framebox{3 2 3 $\cdots$ 3 2 1}  & \framebox{3}\\

\framebox{2 1 2 $\cdots$ 2 1 \, 3} & \framebox{2 1 2 $\cdots$ 2 1 3} & \cdots & \framebox{2 1 2 $\cdots$ 2 1 3} & \framebox{2} \\

\framebox{1 2 1 $\cdots$ 1 {\color{red}{2}} \, 1} & \color{gray}{\framebox{3 2 3 $\cdots$ 3 2 1}}  & \color{gray}{\cdots} & 
 \color{gray}{\framebox{3 2 3 $\cdots$ 3 2 1}} & \color{gray}{\framebox{3}} \\

 \color{gray}{\framebox{2 3 2 $\cdots$ 2 1 $\mid$ 3}} & & & &
\end{array}
\right]
\end{array}
\]
\medskip

The vertex $v_{-1}$ highlighted above is the only vertex labeled 2 whose neighbors are all labeled 1. Because this vertex is unique, the coloring is distinguished by dihedral symmetries. 

\bigskip

\textsc{Subcase 3: Let $\ell > 1$ and $m$ odd }.

\bigskip

Use the same arrangement as in Subcase 2 (with $\ell = 1$) with extended $E_1$ and $E_2$ blocks. 

\[
\begin{array}{c}
\overbrace{\hspace{4.4 cm}}^{r}\hspace{.3 cm} \overbrace{\hspace{2.4 cm}}^{r}
\hspace{1 cm}
\overbrace{\hspace{2.4 cm}}^{r} \hspace{.3 cm} \overbrace{\hspace{2.4 cm}}^{\ell < r \text{ (odd)}} \\

\left[\begin{array}{ccccc}
\framebox{3 2 3 $\cdots$ 2 3 2 \, 3 $\cdots$ 3 2 1} & \framebox{3 2 3 $\cdots$ 3 2 1} & \cdots & \framebox{3 2 3 $\cdots$ 3 2 1}  & \framebox{3 1 3 $\cdots$ 3 1 3}\\

\framebox{2 1 2 $\cdots$ 1 2 1 \, 2 $\cdots$ 2 1 3} & \framebox{2 1 2 $\cdots$ 2 1 3} & \cdots & \framebox{2 1 2 $\cdots$ 2 1 3} & \framebox{2 3 2 $\cdots$ 2 3 2} \\

\framebox{{\color{red}{1 2 1 $\cdots$ 2 1 2 \, 1 $\cdots$ 1 2 1}}} & \color{gray}{\framebox{3 2 3 $\cdots$ 3 2 1}}  & \color{gray}{\cdots} & 
 \color{gray}{\framebox{3 2 3 $\cdots$ 3 2 1}} & \color{gray}{\framebox{3 2 3 $\cdots$ 3 2 3}}\\

 \color{gray}{\framebox{2 3 2 $\cdots$ 3 2 1 $\mid$ 3 $\cdots$ 3 1 3}} & & & & \\

 \underbrace{\hspace{2.3 cm}}_{\ell - r\text{ (even)}}\hspace{.1 cm} \underbrace{\hspace{1.7 cm}}_{\ell\text{ (odd)}} &&&&
\end{array}
\right]
\end{array}
\]

\medskip

The vertices $v_{2k+1}$ to $v_n$ highlighted above represent a unique labeled string of vertices (alternating labels 1 and 2) of length $r$ as every other block contains a 3 and is immediately proceeded or followed by a vertex labeled 3. Because this string of vertices is unique, the coloring is distinguished by dihedral symmetries.

\textsc{Subcase 4: Let $\ell > 1$ and $m$ even and hence $\ell$ even }.

\bigskip

Use the blocks as given below.  

\[
\begin{array}{c}
\overbrace{\hspace{4.4 cm}}^{r}\hspace{.3 cm} \overbrace{\hspace{2.4 cm}}^{r}
\hspace{1 cm}
\overbrace{\hspace{2.4 cm}}^{r} \hspace{.3 cm} \overbrace{\hspace{2.4 cm}}^{\ell < r \text{ (even)}} \\

\left[\begin{array}{ccccc}
\framebox{1 3 1 $\cdots$ 3 1 \, 3 1 $\cdots$ 1 3 1} & \framebox{3 2 3 $\cdots$ 3 2 1} & \cdots & \framebox{3 2 3 $\cdots$ 3 2 1}  & \framebox{3 1 3 $\cdots$ 1 3 1}\\

\framebox{3 2 3 $\cdots$ 2 3 \, 2 3 $\cdots$ 3 2 3} & \framebox{2 1 2 $\cdots$ 2 1 3} & \cdots & \framebox{2 1 2 $\cdots$ 2 1 3} & \framebox{2 3 2 $\cdots$ 3 2 3} \\

\framebox{{\color{red}{2 1 2 $\cdots$ 1 2 \, 1 2 $\cdots$ 2 1 2}}} & \color{gray}{\framebox{1 3 1 $\cdots$ 1 3 1}}  & \color{gray}{\cdots} & 
 \color{gray}{\framebox{3 2 3 $\cdots$ 3 2 1}} & \color{gray}{\framebox{3 2 3 $\cdots$ 2 3 2}}\\

\color{gray}{\framebox{3 2 3 $\cdots$ 2 1 $\mid$ 3 1 $\cdots$ 1 3 1}} & & & & \\

\underbrace{\hspace{2 cm}}_{\ell - r\text{ (odd)}}\hspace{.1 cm} \underbrace{\hspace{2 cm}}_{\ell\text{ (even)}} &&&&
\end{array}
\right]
\end{array}
\]

\medskip

The vertices $v_{2k+1}$ to $v_n$ highlighted above represent a unique labeled string of vertices (alternating labels 2 and 1) of length $r$ as every other block contains a 3 and is immediately proceeded or followed by a vertex labeled 3. Because this string of vertices is unique, the coloring is distinguished by dihedral symmetries.

\bigskip

\textsc{Case 2: Let $n/3<k<n/2 -1$ such that $n=2k+r$ hence $r$ odd and $ r > k/2$}. 
\medskip

It can be verified that the following arrangement of these sequences yield a proper coloring $C$ of $C_n(1, k)$  that is also distinguishing.

\medskip 

\textsc{Subcase 1: Let $2r-k < k-r$ }.

\bigskip

\[
\begin{array}{c}
\overbrace{\hspace{4cm}}^{k-r \text{ (odd)}}\hspace{.4 cm} \overbrace{\hspace{3.3 cm}}^{2r-k \text{ (even)}}
\hspace{.4 cm}
\overbrace{\hspace{3 cm}}^{k-r \text{ (odd)}}\\

\left[\begin{array}{ccc}
\framebox{1 2 $\cdots$ 1 2 1 2 $\cdots$ 1 2 3} & \framebox{1 3 $\cdots$ 1 3 1 $\cdots$ 1 3} & \framebox{1 2 1 2 1 $\cdots$ 1 2 1}\\

\framebox{2 3 $\cdots$ 2 3 2 3 $\cdots$ 2 3 2} & \framebox{3 1 $\cdots$ 3 1 3 $\cdots$ 3 1} & \framebox{3 1 2 1 2 $\cdots$ 2 1 2} \\

\framebox{{\color{red}{3}} 1 $\cdots$ 3 1 3 1 $\cdots$ 3 1 3} & \framebox{1 3 $\cdots$ 1 3 1 $\cdots$ 1 3} & \color{gray}{\framebox{1 2 1 2 1 $\cdots$ 1 2 3} } \\ 
\underbrace{\color{gray}{\framebox{1 3 $\cdots$ 1 3  1 2 $\cdots$ 1 2 1}}}_{\text{ if } \, 2r-k \, < \, k-r}& \color{gray}{\framebox{2 1 $\cdots$ 2 1 2 $\cdots$ 2 1} } &
\end{array}
\right]
\end{array}
\]

\medskip

The vertex $v_{2k+1}$ highlighted above is the only vertex labeled 3 with two neighbors labeled 2 and two neighbors labeled 1. Because this vertex is unique, the coloring is distinguished by dihedral symmetries. 

\bigskip

\textsc{Subcase 2: Let $2r-k > k-r$ }.

\bigskip

\[
\begin{array}{c}
\overbrace{\hspace{4cm}}^{k-r \text{ (odd)}}\hspace{.4 cm} \overbrace{\hspace{3.3 cm}}^{2r-k \text{ (even)}}
\hspace{.4 cm}
\overbrace{\hspace{3 cm}}^{k-r \text{ (odd)}}\\

\left[\begin{array}{ccc}
\framebox{1 2 $\cdots$ 1 2 1 2 $\cdots$ 1 2 3} & \framebox{1 3 $\cdots$ 1 3 1 $\cdots$ 1 3} & \framebox{1 2 1 2 1 $\cdots$ 1 2 1}\\

\framebox{2 3 $\cdots$ 2 3 2 3 $\cdots$ 2 3 2} & \framebox{3 1 $\cdots$ 3 1 3 $\cdots$ 3 1} & \framebox{3 1 2 1 2 $\cdots$ 2 1 2} \\

\framebox{{\color{red}{3}} 1 $\cdots$ 3 1 3 1 $\cdots$ 3 1 3} & \framebox{1 3 $\cdots$ 1 3 1 $\cdots$ 1 3} & \color{gray}{\framebox{1 2 1 2 1 $\cdots$ 1 2 3} } \\ 
\color{gray}{\framebox{1 3 $\cdots$ 1 3  1 3 $\cdots$ 1 3 1}} & \underbrace{\color{gray}{\framebox{3 1 $\cdots$ 3 1 2 $ \cdots$ 2 1} }}_{\text{ if } 2r-k \, > \,  k -r} &
\end{array}
\right]
\end{array}
\]

\medskip

The vertex $v_{2k+1}$ highlighted above is the only vertex labeled 3 with two neighbors labeled 2 and two neighbors labeled 1. Because this vertex is unique, the coloring is distinguished by dihedral symmetries. 

\bigskip

\textsc{Case 3: $n=pk+r$ for some $p\geq 3$}.

\medskip

We will similar blocks as those used in Section 7.3, given below:

\hfil \begin{tabular}{ccccccccccc}
    $B_1$ & $=$ & ($1$, & $2$, & $3$, & $2$, & $3$, & $\cdots$, & $2$, & $3$, & $1$); \\
    $B_2$ & $=$ & ($2$, & $3$, & $1$, & $3$, & $1$, & $\cdots$, & $3$, & $1$, & $2$); \\
    $B_3$ & $=$ & ($3$, & $1$, & $2$, & $1$, & $2$, & $\cdots$, & $1$, & $2$, & $3$); \\
    $B_4$ & $=$ & ($3$, & $1$, & $3$, & $1$, & $3$, & $\cdots$, & $1$, & $3$, & $1$); \\
    $B'$ & $=$ & ($2$, & $1$, & $2$, & $1$, & $2$, & $\cdots$, & $2$, & $1$, & $2$).
\end{tabular} \hfil
\bigskip

The blocks of colors $B_1$, $B_2$, $B_3$ are constructed in such a way that as long as blocks are not repeated, adjacent vertices will not have the same color. It is also the case between $B_4$ and $B_2$. We give a coloring $C$ of the vertices of $C_n(1,k)$ using colors 1, 2, and 3 as follows: assign colors to $v_1,v_2, \ldots, v_k$ using $B_4$. Next, starting with $B_2$, assign colors to $v_{k+1}, v_{k+2}, \ldots, v_{(q-1)k}$ by alternating the use of sequences $B_2$, $B_3$ on successive collections of $k$ consecutively-indexed vertices. Assign colors to $v_{(q-1)k+1}, v_{(q-2)k+2},\ldots ,v_{qk}$ using $B_1$. Lastly, use the shortened block $B'$ of length $r$ on the remaining vertices, noting that if $r=1$ we are left with the vertex $v_n$ labeled 2. 

 Use the blocks as given below, the first is where $p$ is odd so the repetition of $B_2$ and $B_3$ ends with a $B_2$, and the latter is where $p$ is even and the repetition ends with a $B_3$.

\[
\begin{array}{cc}

\begin{array}{cc}

\begin{array}{c}
{\color{red}{B_4}} \\
B_2 \\
B_3 \\
\vdots \\
B_2 \\
B_1 \\
B'
\end{array}

\begin{array}{cc}
\overbrace{\hspace{2.9 cm}}^{r \text{ (odd)}}\hspace{.3 cm} \overbrace{\hspace{2.2 cm}}^{k-r \text{ (odd)}}\\

\left[\begin{array}{cc}
\framebox{{\color{red}{3 1 3 1 3 $\cdots$ 3 1 3}}} & \framebox{{\color{red}{1 3 $\cdots$ 1 3 1}}}\\

\framebox{2 3 1 3 1 $\cdots$ 1 3 1} & \framebox{3 1 $\cdots$ 3 1 2}\\

\framebox{3 1 2 1 2 $\cdots$ 2 1 2} & \framebox{1 2 $\cdots$ 1 2 3}\\

$\vdots$ & $\vdots$\\


\framebox{2 3 1 3 1 $\cdots$ 1 3 1} & \framebox{3 1 $\cdots$ 3 1 2}\\

\framebox{1 2 3 2 3 $\cdots$ 3 2 3} & \framebox{2 3 $\cdots$ 2 3 1}\\

\framebox{2 1 2 1 2 $\cdots$ 2 1 2} & \color{gray}{\framebox{3 1 $\cdots$ 3 1 3}}\\

\color{gray}{\framebox{1 3 1 3 1 $\cdots$ 1 3 1}} &

\end{array}
\right]
\end{array}

\end{array}

&

\begin{array}{cc}

\begin{array}{c}
{\color{red}{B_4}} \\
B_2 \\
B_3 \\
\vdots \\
B_3 \\
B_1 \\
B'
\end{array}

\begin{array}{cc}
\overbrace{\hspace{2.9 cm}}^{r \text{ (odd)}}\hspace{.3 cm} \overbrace{\hspace{2.2 cm}}^{k-r \text{ (odd)}}\\

\left[\begin{array}{cc}
\framebox{{\color{red}{3 1 3 1 3 $\cdots$ 3 1 3}}} & \framebox{{\color{red}{1 3 $\cdots$ 1 3 1}}}\\

\framebox{2 3 1 3 1 $\cdots$ 1 3 1} & \framebox{3 1 $\cdots$ 3 1 2}\\

\framebox{3 1 2 1 2 $\cdots$ 2 1 2} & \framebox{1 2 $\cdots$ 1 2 3}\\

$\vdots$ & $\vdots$\\


\framebox{3 1 2 1 2 $\cdots$ 2 1 2} & \framebox{1 2 $\cdots$ 1 2 3}\\

\framebox{1 2 3 2 3 $\cdots$ 3 2 3} & \framebox{2 3 $\cdots$ 2 3 1}\\

\framebox{2 1 2 1 2 $\cdots$ 2 1 2} & \color{gray}{\framebox{3 1 $\cdots$ 3 1 3}}\\

\color{gray}{\framebox{1 3 1 3 1 $\cdots$ 1 3 1}} &

\end{array}
\right]
\end{array}

\end{array}
\end{array}
\]
\medskip

The vertices in the block $B_4$ in both cases above represent a unique string of $r$ vertices with alternating labels 3 and 1. Since every other block contains multiple vertices labeled 2, the string is unique and the coloring is distinguished by dihedral symmetries. 

\end{proof}

\section{Graphs $C_n(1,k)$ where $k^2 \equiv \pm 1 \pmod{n}$}
\label{sec: k^2 = pm 1}

\begin{figure}
    \centering

    \begin{tikzpicture}[scale=1.5]
     \node[gray]  at (90:2.3) {$v_1$};
     \node[gray]  at (62.3:2.3) {$v_2$};
     \node[gray]   at (34.6:2.3) {$v_3$};
     \node[gray]   at (6.9:2.3) {$v_4$};
     \node[gray]   at (339.2:2.3) {$v_5$};
     \node[gray]   at (311.5:2.3) {$v_6$};
     \node[gray]   at (283.8:2.3) {$v_7$};
     \node[gray]   at (256.2:2.3) {$v_8$};
     \node[gray]   at (228.5:2.3) {$v_9$};
     \node[gray]   at (200.8:2.3) {$v_{10}$};
     \node[gray]  at (173.1:2.3) {$v_{11}$};
     \node[gray]   at (145.4:2.3) {$v_{12}$};
     \node[gray] at (117.7:2.3) {$v_{13}$};
     \node (0) [draw,shape=circle, fill=black, scale=.4] at (90:2) {};
     \node (1) [draw,shape=circle, fill=black, scale=.4] at (62.3:2) {};
     \node (2) [draw,shape=circle, fill=black, scale=.4] at (34.6:2) {};
     \node (3) [draw,shape=circle, fill=black, scale=.4] at (6.9:2) {};
     \node (11) [draw,shape=circle, fill=black, scale=.4] at (145.4:2) {};
     \node (12)  [draw,shape=circle, fill=black, scale=.4] at (117.7:2) {};
     \node  (10) [draw,shape=circle, fill=black, scale=.4] at (173.1:2) {};
     \node (9) [draw,shape=circle, fill=black, scale=.4] at (200.8:2) {};
     \node (8) [draw,shape=circle, fill=black, scale=.4] at (228.5:2) {};
     \node (7) [draw,shape=circle, fill=black, scale=.4] at (256.2:2) {};
     \node (6) [draw,shape=circle, fill=black, scale=.4] at (283.8:2) {};
     \node (5) [draw,shape=circle, fill=black, scale=.4] at (311.5:2) {};
     \node (4)  [draw,shape=circle, fill=black, scale=.4] at (339.2:2) {};
     \draw (0)--(1)--(2)--(3)--(4)--(5)--(6)--(7)--(8)--(9)--(10)--(11)--(12)--(0);
     \draw (0)--(5)--(10)--(2)--(7)--(12)--(4)--(9)--(1)--(6)--(11)--(3)--(8)--(0);
     \node[blue]  at (90:2.6) {1};
     \node[red]  at (62.3:2.6) {2};
     \node[olive]   at (34.6:2.6) {3};
     \node[blue]   at (6.9:2.6) {1};
     \node[red]   at (339.2:2.6) {2};
     \node[olive]   at (311.5:2.6) {3};
     \node[blue]   at (283.8:2.6) {1};
     \node[red]   at (256.2:2.6) {2};
     \node[olive]   at (228.5:2.6) {3};
     \node[blue]   at (200.8:2.6) {1};
     \node[red]  at (173.1:2.6) {2};
     \node[olive]   at (145.4:2.6) {3};
     \node[violet] at (117.7:2.6) {4};
\end{tikzpicture}

    \caption{A proper distinguishing coloring of $C_{13}(1,5)$ with 4 colors}
    \label{C_13(1,5)}
\end{figure}
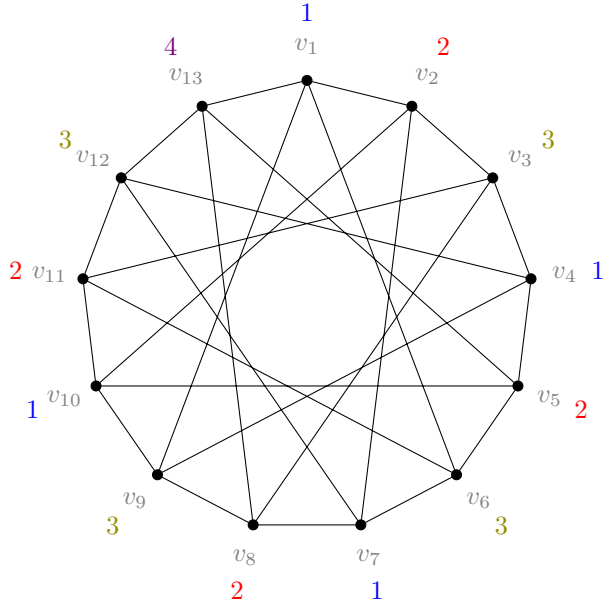

Theorem~\ref{thm: Aut when k^2 equiv pm 1} and other results in Section~\ref{sec: automorphism groups} showed that that when $k^2 \equiv \pm 1 \pmod{n}$, the graphs $C_n(1,k)$ are dart-transitive, but any automorphism not corresponding to a dihedral symmetry on the Hamiltonian cycle comprised of edges in $E_1$ corresponds to a symmetry that maps the edges of $E_k$ to this Hamiltonian cycle while mapping the edges of $E_1$ to those in $E_k$. It follows that, given a coloring $C$ that destroys all dihedral symmetries of $C_n(1,k),$ if $C$ also fixes an edge, then $C$ destroys all non-dihedral symmetries as well. We begin with the case of $C_{13}(1,5)$ to illustrate this idea. 

\begin{prop} \label{prop: C13 1 5}
$\chi_D(C_{13}(1,5))= \chi(C_{13}(1, 5))= 4$.
\end{prop}
\begin{proof}
    Consider the following coloring: starting at $v_1$, assign colors to $v_1, v_2, \cdots, v_{12}$ by repeatedly using the block of colors $B= (1, 2,3)$, and assign color 4 to $v_{13}.$  We can see in Figure~\ref{C_13(1,5)} that the coloring is proper. Moreover, it is distinguishing since the uniquely colored vertex $v_{13}$ is fixed and the edge $v_{12}v_{13}$ is also fixed, as it is the only edge with endpoints colored 4 and 3. 
\end{proof}

The graph $C_{15}(1,4)$ will form an exception for our next main result, so we establish its distinguishing chromatic number here. Unlike with $C_{13}(1,5)$ the distinguishing chromatic number will be greater than the chromatic number, which is 3, so more effort is required for the lower bound. We begin with a lemma that incidentally shows that the distinguishing coloring of $C_8(1,4)$ in Section~\ref{sec: trivalent} was unique up to the names of the colors.
\begin{lem} \label{lem: Cn(1,4) pre-labeling}
    Let $n$ be an integer with $n \geq 8$. In any distinguishing 3-coloring of $C_n(1,4)$, there exist 8 consecutively-indexed vertices receiving colors in the pattern $C,A,B,C,B,C,A,B$, where $\{A,B,C\}$ is the set of colors used.
\end{lem}
\begin{proof}
    Observe first that in any proper coloring there must be at least one index $i$ such that $v_i,v_{i+1},v_{i+2}$ receive three distinct colors; if this is not the case, then $C_n(1,4)$ would be properly colored with just two colors, and this is a contradiction, since $v_0,v_1,v_2,v_3,v_4$ are the vertices of a 5-cycle in $C_n(1,4)$.
			
    Furthermore, we claim that in any \emph{distinguishing} proper coloring of $C_n(1,4)$ using three colors, there must be an index $j$ such that $v_j,v_{j+1},v_{j+2}$ do \emph{not} receive three distinct colors. Indeed, if no such $j$ existed, then the proper coloring would consist of the pattern $A,B,C$ repeated around the vertices, requiring that $n$ be a multiple of 3 and that the map $v_k \mapsto v_{k+3}$ be a color-preserving automorphism of $C_n(1,4)$, a contradiction.
			
    It follows that there must be an index $i$ such that $v_i,v_{i+1},v_{i+2}$ are colored with three distinct colors, but $v_{i+1},v_{i+2},v_{i+3}$ are not. By symmetry, we may suppose that $i=0$ and that the labels on $v_{0},v_1,v_2,v_3$ are $A,B,C,B$. Since $v_4$ is adjacent to both $v_3$ and $v_0$, the color on $v_4$ must be $C$. Likewise, the colors on $v_5$ and $v_6$ must be $A$ and $B$, respectively, and the color on $v_{n-1}$ (which is adjacent to both $v_0$ and $v_3$) must be $C$. This completes the proof.
\end{proof}

\begin{thm}
    $\chi_D(C_{15}(1,4)) = 4$
\end{thm}
\begin{proof}
    We first show that no distinguishing proper coloring can use 3 colors. Suppose to the contrary that some coloring $c$ does. By Lemma~\ref{lem: Cn(1,4) pre-labeling}, we may suppose that $v_{14},v_0,\dots,v_{6}$ are respectively labeled with $C,A,B,C,B,C,A,B$. Since $v_{10}$ is adjacent to $v_6$ and to $v_{14}$, we have $c(v_{10})=A$. We now proceed by cases on the colors of $v_7,v_8,v_9,v_{11},v_{12},v_{13}$, showing that any proper 3-coloring admits a color-preserving automorphism.
	
    \medskip
    \textsc{Case: $c(v_9)=B$ or $c(v_{11})=C$}. Suppose first that $c(v_9)=B$. Then $c(v_8)=A$ and $c(v_{13})=A$, since each vertex has a neighbor already colored with $B$ and with $C$. Then $c(v_7)=C$. and $c(v_{12})=C$ and consequently $c(v_{11})=B$. At this point all the vertices have been colored, and one can verify that the permutation $(v_1 \ v_{11})(v_2 \ v_7)(v_4 \ v_{14})(v_5 \ v_{10})(v_8 \ v_{13})$ is an automorphism of $C_{15}(1,4)$ preserving the coloring $c$.
			
    Note that the starting pattern of colors appears on vertices $v_6,v_5,v_4,v_3,v_2,v_1,v_0,v_{14}$, in this order, if the roles of colors $B$ and $C$ are switched, so the case where $c(v_{11})=C$ follows exactly the same argument as above to conclude that there is an automorphism of $C_{15}(1,4)$ preserving the coloring.
			
    \medskip
    Assume henceforth that $c(v_9) \neq B$ and $c(v_{11})\neq C$. This forces $c(v_9)=C$ and $c(v_{11})=B$.
			
    \medskip
    \textsc{Case: $c(v_8)=A$ or $c(v_{12})=A$}. By the symmetry argument above, it suffices to suppose that $c(v_8)=A$. It follows that $c(v_7)=C$ and that $c(v_{11})=B$ and $c(v_{12})=C$. Now, if $c(v_{13})=A$, we get the coloring-preserving automorphism $(v_1 \ v_{11})(v_2 \ v_7)(v_4 \ v_{14)}(v_5 \ v_{10})(v_8 \ v_{13})$. If $c(v_{13})=B$, then we get the color-preserving automorphism $(v_0 \ v_{10})(v_1 \ v_6)(v_3 \ v_{13})(v_4 \ v_9)(v_7 \ v_{12})$.
						
    \medskip
    Having handled the cases above, assume henceforth that $c(v_8) \neq A$ and $c(v_{12})\neq A$. This forces $c(v_8)=B$ and $c(v_{12})=C$.
			
    \medskip
    \textsc{Case: $c(v_7)=A$ or $c(v_{13})=A$}. By the prior symmetry argument, it suffices to assume that $c(v_7)=A$. No matter what color $v_{13}$ receives, we have the color-preserving symmetry $(v_0 \ v_5)(v_{2} \ v_{12})(v_3 \ v_8)(v_6 \ v_{11})(v_9 \ v_{14})$.
			
    \medskip
    In light of these cases, we see that $c(v_7) \neq A$ and $c(v_{13})\neq A$. This forces $c(v_7)=C$ and $c(v_{13})=B$. However, then the map $v_{i} \mapsto v_{i+5}$ yields a color-preserving automorphism.
			
    \medskip
    This concludes the necessary cases; we have shown that every proper 3-coloring of $C_{15}(1,4)$ is not distinguishing. On the other hand, $\chi(C_{15}(1,4)) = 3$, and the coloring that alternates colors $A,B,C$ around the vertices is a palindrome-free coloring of $C_{n}(1,4)$ that satisfies Theorem 6.3 of the paper. The Theorem then implies that $\chi_D(C_{15}(1,4))=4$.
\end{proof}

With the exceptions dealt with, we now prove the general result. We will prove the colorings introduced in Section 7 will also be distinguished by any symmetry that swaps the edges comprising of a Hamilton cycle of $E_1$ and those comprising a Hamilton cycle of $E_K$. We do this by finding an edge that is fixed by the coloring. Note that this proof is more than is necessary as most graphs covered in Section 7 will not have an edge swapping symmetry like those discussed in this section, but will be sufficient to prove our result. 

\begin{thm} \label{thm: nondihedral}
For $k^2 \equiv \pm 1 \pmod{n}$ with $k>3$, $\chi_D(C_n(1,k)) = 3,$ except for $C_{13}(1, 5)$ and $C_{15}(1, 4)$ where the distinguishing chromatic number is 4.
\end{thm}

\begin{proof}

Like in the dihedral case, we proceed on a case-by-case basis. Note that the corresponding theorem and case in Section 7 will be listed for reference at the beginning of each case. 

\bigskip

\textsc{Case: $n$ is odd and $k$ is odd (referencing Theorem 7.6 in Section 7.4)} 

\medskip

\textsc{Subcase 1: $n \geq 3k$ (referencing Theorem 7.6 Case 1)} 

\medskip

Consider the distinguishing coloring $C$ given in the case of its dihedral analogue. It can be easily verified that $v_k$ is the only vertex labeled 1 with exactly two neighbors colored 2 and the others colored 3. Similarly, $v_n$ is the only vertex colored 2 with exactly two neighbors colored 1 and the others colored 3. By construction, both vertices $v_0$ and $v_k$ are fixed and, hence, the edge $v_0v_k$ is fixed. 

\medskip

\textsc{Subcase 2: $n= 2k + r$ (referencing Theorem 7.6 Case 2)}

\medskip

Suppose that $\displaystyle 2r \geq k+1.$ We consider the same subcases together with the corresponding distinguishing colorings discussed in the dihedral section. If $2r-k > k-r$ (previously referred as SUBCASE 4), the vertex $v_{k+1}$ is the only vertex labeled 3 with exactly two of its neighbors labeled 2 and the others labeled 1. Moreover, the neighbors labeled 1, namely $v_k$ and $v_{2k+1}$ have distinct neighborhoods. Hence, the edge $v_kv_{k+1}$ is fixed. If $2r-k < k-r$ with $2r-k > 1$ (previously referred as SUBCASE 1), the vertex $v_{k-r+2}$ is labeled 1 with all neighbors labeled 3. Switch the color of $v_{k-r+2}$ to 2. Thus, $v_{k-r+2}$ becomes the only vertex labeled 2 that has all its neighbors labeled 3. Most importantly, the coloring is still distinguishing with respect to the dihedral group. Moreover, since $v_{-r+2}$ has a unique neighborhood among the neighbors of $v_{k-r+2},$ the edge $v_{k-r+2}v_{-r+2}$ is fixed. If $2r-k < k-r$ with $2r-k = 1$ and $2k-3r > 1$ (previously referred as SUBCASE 2), recall that the vertex $v_{k+1}$ is the only vertex labeled 3 with two of its neighbors labeled 2 and two labeled 1. Moreover, the neighbors labeled, namely $v_{2k+1}$ and $v_k,$ have distinct neighborhoods. Thus, the edge $v_{k+1}v_k$ is fixed. Lastly, if $2r-k < k-r$ with $2r-k = 2k-3r= 1$ (previously referred as SUBCASE 3), then we have the special graph $C_{13}(1, 5)$ (see Theorem~\ref{thm: chromatic num}). We saw at the beginning of this section a construction of a coloring that is distinguishing. 

\medskip

\textsc{Subcase 3: $n= 2k + r$ (referencing Theorem 7.6 Case 3)}

\medskip

Suppose $\displaystyle 2r \leq k-1.$ which implies $k = pr+\ell.$ Again, we consider the same subcases discussed in the section dedicated to the dihedral group. If $\ell = 0$ (previously referred as SUBCASE 1), the vertex $v_{k+1}$ is the only 2-colored vertex whose neighbors are all colored 1. Thus, $v_{k+1}$ is fixed. Moreover, $v_2$ is colored 3 and all its neighbors are colored 1. Switch the color of $v_2$ to 2.  This moves makes the neighborhood of $v_1$ distinct from the other neighbors of $v_{k+1}$ while the coloring remains distinguishing with respect to the dihedral group. Moreover, we have $v_1v_{k+1}$ as a fixed edge. If $\ell$ is even and $\ell \neq 0$ (previously referred as SUBCASE 4), the vertex $v_{k+1}$ is fixed by virtue of being the only 3-colored vertex with exactly two neighbors colored 2 and the others colored 1. Moreover, since $v_1$ is uniquely colored among the neighbors of $v_{k+1},$ we thus have $v_{k+1}v_1$ as a fixed edge. If $\ell$ is odd (previously referred as SUBCASES 2 and 3), the vertex $v_{k+r+2}$ is colored 1 with all its neighbors colored 2 and vertex $v_{k+r+1}$ is colored 2 with exactly 3 neighbors colored 3 and the other one, $v_{k+r+2},$ colored 1. Fix vertex $v_{k+r+1}$ by switching the color $v_{k+r+2}$ from 1 to 3.  The switch makes $v_{k+r+1}$ the only 2-colored vertex whose neighbors are all colored 3. Most importantly, the coloring remains distinguishing with respect to the dihedral group. Moreover, $v_{k+r}$ has a distinct neighborhood among the neighbors of $v_{k+r+1}.$ Therefore, the edge $v_{k+r}v_{k+r+1}$ is fixed.

\bigskip

\textsc{Case: $n$ is odd and $k$ is even (referencing Theorem 7.7 in Section 7.5)}

\medskip

\textsc{Subcase 1: $n \geq 3k$ (referencing Theorem 7.7 Case 3)}

\medskip

Let $n= qk+r,$ where $q \geq 3.$ Since $n$ is odd and $k$ is even, $r$ is also odd and $r\neq 0.$ 

First, suppose that $q$ is even. Consider the coloring $C$ given in the case of its dihedral analogue; which is the case where the block $B_1$ is preceded by a block $B_3.$ If $r < k-1,$ then the vertex $v_1$ and $v_2$ are labeled 3 and 1, respectively. Switch the color of $v_1$ to 1 and the color of $v_2$ to 2. Note that the modified $C$ is still proper. Furthermore, the moves makes $v_1$ a vertex labeled 1 with all its neighbors colored 2. Thus, $v_1$ is fixed since there exists no other vertex labeled 1 with all its neighbors labeled 3. Thus, the modified $C$ is distinguishing with respect to the dihedral group. Moreover, $v_2$ has a distinct neighborhood among the neighbors of $v_1.$ In particular, $v_2$ has three neighbors labeled 3. Therefore, the edge $v_1v_2$ is also fixed. Hence, the modified $c$ is distinguishing with respect to nondihedral group. If $r = k-1 \geq 3,$ the vertex $v_{qk+1}$ (first vertex in block $B'$) is colored 2 with all its neighbors colored 1. Now, fix vertex $v_{qk+1}$ by switching the colors of all possible like-vertices (a vertex that is labeled 2 with all neighbors labeled 1) to 3; for example, such vertex can be found in a Block $B_3$ squeezed between two blocks $B_2.$ Therefore, the modified $C$ is distinguishing with respect to the dihedral group. Moreover, the vertex $v_{qk+2}$ has a unique neighborhood among the neighbors of $v_{q+1}.$ In particular, at least three of its neighbors are labeled 2. Therefore, the edge $v_{q+1}v_{q+2}$ is fixed. Hence, the modified $C$ is distinguishing with respect to nondihedral group.

Now suppose that $q$ is odd. Consider the coloring $C$ given in the case of its dihedral analogue; which is the case where the block $B_1$ is preceded by a block $B_2.$ If $q >3,$ there exists at least a block $B_3$ squeezed between two blocks $B_2$ and the first vertex of the block $B_1$, namely $v_{(q-1)k+1},$ is labeled 1 with all its neighbors labeled 2. Since no other vertex labeled 1 has the same neighborhood, $v_{(q-1)k+1}$ is thus fixed. Moreover, the vertex  $v_{qk+1}$ has a distinct neighborhood among the neighbors of $v_{(q-1)k+1}$. In particular, all the neighbors of $v_{qk+1}$ are labeled 1. Therefore, the edge $v_{(q-1)k+1}v_{qk+1}$ is fixed. Hence, $C$ is distinguishing with respect to the nondihedral group. If $q=3$ and $k \neq 4,$ the vertex $v_{k+3}$ is labeled 1 with all its neighbors labeled 3. Switch its color to 2. This moves makes $v_{k+3}$ a vertex labeled 2 with all its neighbors labeled 3. Thus, $v_{k+3}$ is fixed since there is no such other vertex and the modified $C$ is distinguishing with respect to the dihedral group. Moreover, $v_{k+4}$ has a distinct neighborhood among the neighbors of $v_{k+3}.$ More specifically, $v_{k+4}$ has exactly two neighbors labeled 2 and two neighbors labeled 1. Therefore, the edge $v_{k+3}v_{k+4}$ is fixed. Hence, the modified $C$ is distinguishing with respect to the nondihedral group. If $q=3$ and $k=4,$ then we have the exceptional graph $C_{15}(1, 4)$.

\medskip

\textsc{Subcase 2: $n=2k+r$ (referencing Section 7.7 Cases 1 and 2)}

\medskip

Suppose $\displaystyle r > k/2.$ Consider the distinguishing coloring $C$ given in the case of its dihedral analogue. Then $v_1$ is fixed since it is the only vertex colored 1 that has excatly two neighbors colored 2 (namely $v_2$ and $v_{k+1}$) and the other colored 3. Since $v_2$ is distinctly colored than $v_{k+1},$ thus $v_1v_2$ is a fixed edge. Hence, the desired result.

Suppose $\displaystyle r \leq k/2,$ which implies $k = pr+\ell.$ Consider the distinguishing coloring $C$ given in the case of its dihedral analogue. The same argument used in the case where both $n$ and $k$ are odd also works here.

\bigskip

\textsc{Case: $n$ is even and $k$ is even (referencing Theorem 7.5 in Section 7.3)}

\medskip

\textsc{Subcase 1: $n \geq 3k$ (referencing Theorem 7.5 Case 1)}

\medskip

Again, consider the coloring $C$ given in the case of its dihedral analogue. Recall that when $r>2$ and if we switch  the color of vertex $v_0$ to 3, we obtain a proper distinguishing coloring with respect to the dihedral group. Moreover, $v_{-1}$ has a distinct neighborhood among the neighbors of $v_0.$ In particular, $v_{-1}$ has three of its neighbors colored 3 while every other neighbor of $v_0$ is adjacent to at most two vertices colored 3.Therefore, the edge $v_{-1}v_0$ is fixed. Hence, the desired result. When $r=2,$ the same argument in the dihedral case holds.

\medskip

\textsc{Subcase 2: $n=2k+r$ (referencing Theorem 7.5 Case 2)}

\medskip

Suppose $r \leq k/2.$ Consider the distinguishing coloring $C$ given in the case of its dihedral analogue. With respect to $C,$ the vertex $v_{k+1}$ is fixed since it is the only vertex colored 3 that has excatly two neighbors colored 1 (namely $v_1$ and $v_{k+2}$) and the other colored 2. Moreover, $v_2$ is colored 2 and all its neighbors are colored 1. Switch the color of $v_2$ to 3.  This moves makes  $v_1$ uniquely colored among the neighbors of $v_{k+1},$ so we have $v_1v_{k+1}$ as a fixed edge.

Similarly when $r > k/2.$

Note that the case where $n$ is even and $k$ is odd is already addressed in Theorem~\ref{thm: n even k odd}.

\end{proof}

\bigskip
END OF PAPER
\bigskip

\end{document}